\documentclass[11pt]{amsart}

\usepackage{amssymb}
\usepackage{amsmath}
\usepackage{amsthm}
\usepackage{graphicx}
\usepackage{macros}
\usepackage[usenames,dvipsnames]{xcolor}
\usepackage{tikz}
\usepackage{subcaption}
\usepackage{tikz}
\usetikzlibrary{arrows.meta, positioning}
\usepackage{float}
\usepackage{color}
\usepackage{hyperref}
\usepackage{mathrsfs}
\hypersetup{
  colorlinks,
  citecolor= Brown,
  linkcolor= link,
  urlcolor= link
}

\usepackage{appendix}
\usepackage{array}
\usepackage{footnote}
\usepackage{booktabs}
\usepackage{hhline}
\makesavenoteenv{tabular}
\usepackage{scrextend}

 \deffootnote{0em}{1.6em}{\enskip}

\definecolor{link}{rgb}{0.18,0.25,0.63}
\definecolor{myred}{rgb}{0.7,0.25,0.2}
\usepackage[top=3cm,bottom=3cm,left=2.5cm,right=2.5cm]{geometry}

\numberwithin{equation}{section}

\DeclareMathOperator*{\supp}{\mathrm{supp}}

\makeatletter
\g@addto@macro{\endabstract}{\@setabstract}
\newcommand{\authorfootnotes}{\renewcommand\thefootnote{\@fnsymbol\c@footnote}}%
\makeatother

\allowdisplaybreaks

\begin{document}
\definecolor{link}{rgb}{0,0,0}
\definecolor{mygrey}{rgb}{0.34,0.34,0.34}
\def\blue #1{{\color{blue}#1}}

 \begin{center}
 \large
  \textbf{Variable step mollifiers and applications} \par \bigskip \bigskip
  
   \normalsize
  \textsc{Michael Hinterm\"uller} \textsuperscript{$\dagger\ddagger$}, \textsc{Kostas Papafitsoros} \textsuperscript{$\dagger$} \textsc{and Carlos N. Rautenberg} \textsuperscript{$^*$}
\let\thefootnote\relax\footnote{
\textsuperscript{$\dagger$}Weierstrass Institute for Applied Analysis and Stochastics (WIAS), Mohrenstrasse 39, 10117, Berlin, Germany
}
\let\thefootnote\relax\footnote{
\textsuperscript{$\ddagger$}Institute for Mathematics, Humboldt University of Berlin, Unter den Linden 6, 10099, Berlin, Germany}
\let\thefootnote\relax\footnote{
\textsuperscript{$^*$}Department of Mathematical Sciences, George Mason University, Fairfax, VA 22030, USA.
}

\let\thefootnote\relax\footnote{
\hspace{3.2pt}Emails: \href{mailto:michael.hintermueller@wias-berlin.de}{\nolinkurl{hintermueller@wias-berlin.de}},
 \href{mailto: papafitsoros@wias-berlin.de}{\nolinkurl{papafitsoros@wias-berlin.de}}, \href{mailto:crautenb@gmu.edu}{\nolinkurl{crautenb@gmu.edu}}}

\end{center}

\begin{abstract}
We consider a mollifying operator with variable step  that, in contrast to the standard mollification, is able to preserve the boundary values of functions. We prove boundedness of the  operator in all basic Lebesgue, Sobolev and BV spaces as well as corresponding approximation results. The results are then applied to extend recently developed theory concerning the density of convex intersections.
\end{abstract}


\definecolor{link}{rgb}{0.18,0.25,0.63}


\section{Introduction}

Consider an open and bounded domain $\om\subseteq \RR^{N}$, and bounded non-negative smooth maps $\rho,\eta:\RR^{N}\to \RR$ such that $\mathrm{supp }(\rho) \subseteq \overline{B_1(0)}:=\{x\in\RR^{N}: |x|\leq 1\}$ with ``$\mathrm{supp }$'' the support set of a function, and $\eta(x)\leq \mathrm{dist}(x,\partial \Omega)$ for all $x\in \Omega$. Then, we define the {\it variable step mollifier} $T$ as follows:
\newtheorem{def_T}{Definition}[section]
\begin{def_T}\label{lbl:def_T}
Let $f\in L_{loc}^{1}(\om)$. We define $Tf(x)$ for $x\in\om$, as
\begin{align}\label{lbl:def_T2}
Tf(x):=M_{\rho}\int_{B_1(0)} \rho (z ) f(x-\eta(x)z)\,dz, & & \text{with} & &M_{\rho}:=\left(\int_{B_1(0)}\rho(y)\,dy\right)^{-1},
\end{align}
with ${B_1(0)}:=\{x\in\RR^{N}: |x| < 1\}$.
\end{def_T}
Note that if $\eta=1$ and if $\rho$ is a smooth function, then \eqref{lbl:def_T2} reduces to the usual mollifier. In this paper, however, we are interested in $\eta(x)\leq \mathrm{dist}(x,\partial \Omega)$ for all $x\in \Omega$, which implies that the domain of integration in the definition of $T$ is always within $\Omega$.

To the best of our knowledge, the idea of such mollifiers with variable step was introduced by Burenkov; see \cite{MR684996, MR1622690} and references therein.  In particular, Burenkov considers mollifiers defined as infinite sums of constant step mollifiers that are localized by means of the partition of identity procedure. Further, the operator in Definition \ref{lbl:def_T} is mentioned on Remark 26 in \cite{MR1622690} and only a few comments are made. A detailed study of such operators, however, is still lacking in the literature. In particular, we focus on identifying conditions on the pair $(\rho, \eta)$ to guarantee boundedness of the operator on Lebesgue, Sobolev, and {bounded variation (BV)} spaces, and the preservation of trace values, respectively. Additionally, we study properties of approximations $T_n$ (where $\eta$ is replaced by $\eta/n$, $n\in\mathbb{N}$) of $T$. Finally, we apply our work to establish results involving the density of convex intersections. For more on applications we refer to \cite{Simon1}, which are partly recalled at the end of this section, for convenience.

In general, $\rho$ denotes a standard smooth function, supported on the unit ball centred at the origin, and
$\eta$ is a non-negative, smooth function in $\om$, with the property that its values and all its derivatives, vanish on the boundary of $\om$. The existence of such a function follows from a classical result of Whitney \cite{whitney1934analytic}; see also Theorem \ref{lbl:eta_existence}. Moreover, the function $M_\rho$ helps to normalize the convolution. The operator $T$ can then be thought of as a mollification process with a mollification step encoded here by the smooth function $\eta$ that satisfies $\eta(x)\leq \mathrm{dist}(x,\partial \Omega)$ for all $x\in \Omega$. For $\eta>0$, one indeed finds that the resulting function $Tf$ is smooth. As a result, the operator $T$ is able to preserve the boundary values of $f$, when regarded in the appropriate sense depending on the regularity of $f$. Note that this is in contrast to the standard mollification procedure, where such a boundary behavior cannot be guaranteed in general. 
In the work by Burenkov \cite{MR1622690}, by using appropriate partitions of unity the approximation by smooth functions while preserving boundary values can also be  achieved. We point out, however, that in our case the definition of the mollifying operator has a more explicit form and is thus easier to handle. This makes the mollifier more suitable for our main applications as discussed below.


\subsection*{Applications}\label{Motivation} 
Let $X$ be a reflexive Banach space of (equivalence classes) of maps $u:\Omega\to \mathbb{R}$ where $\Omega\subset \mathbb{R}^N$ is a domain with Lipschitz boundary, and such that strong convergence of a sequence in $X$ implies the existence of a subsequence converging pointwise, e.g., $X\in \{L^2(\Omega), H_0^{1}(\Omega),\ldots\}$. Let $F:X\to \mathbb{R}$ be continuous, coercive and sequentially weakly lower semicontinuous, and $K$ be defined as
\begin{equation}\label{Kalpha}
K:=\{w\in X: |w(x)|\leq \alpha(x) \text{ for almost every (a.e.) } x\in\Omega\},
\end{equation}
where the {bound} $\alpha$ satisfies
\begin{equation}\label{eq:alpha}
\alpha\in C(\Omega): \quad \alpha(x)>0,\: \text{ for }\: x\in \Omega, \quad\text{and}\quad \alpha(x_n)\to 0, \:\text{ for }\:x_n\to x\in \Gamma\subseteq \partial \Omega,
\end{equation}
for some non-empty subset $\Gamma$. Hence, there is no $\epsilon>0$ such that $|\alpha(x)|\geq \epsilon$ for all $x\in \Omega$. We consider the optimization problem
\begin{equation}\label{eq:OrigP}\tag{$\mathrm{P}$}
\min_{u\in X} F(u), \quad \text{subject to }\quad u\in K,
\end{equation}
which has a solution given the properties of $X$, $F$, and the convex nature of $K$. Such optimization problems, where the bound is spatially {dependent} occur in many places including recent work in mathematical imaging and other optimization problems with partial differential equations (PDE) constraints; see, e.g., \cite{Simon2,hint_rau,hintermuller2017analytical,hintermuller2017optimal,hintermuller2006path, hintermuller2018function, Antil_2018,Antil_2019,Ceretani_2018} and references therein. 

For the development of efficient infinite dimensional algorithms, and finite dimensional approximation (discretization) methods, it is common to regularize the constraint $K$ via the so-called Moreau-Yosida regularization, and replace \eqref{eq:OrigP} by a sequence of problems of the type
\begin{equation}\label{eq:AproxP}\tag{$\mathrm{P}_n$}
\min_{u\in X} F_n(u):= F(u) +\frac{\gamma_n}{2}\left(\inf_{w\in K}\|u-w\|_X\right)+G_n(u),
\end{equation}
where $\gamma_n\to\infty$, and $G_n$ has, in general, the following property: For some dense subspace $Y$ of $X$, we have that $G_n(w)=+\infty$ if $w\notin Y$. There are two important examples fitting into this framework:

\textbf{a. Tikhonov regularization}. Let $Y$ be a dense subspace  of $X$ consisting of more regular functions than those in $X$. An example is $Y=H_0^1(\Omega)$ and $X=L^2(\Omega)$.  Consider
\begin{equation*}
G_n(w):=\frac{1}{\delta_n}\|w\|_Y^p,
\end{equation*}
where $\delta_n\to\infty$, and $p\geq 1$. Note that in this setting the solution to \eqref{eq:AproxP} approximating the solution of \eqref{eq:OrigP} exhibits higher regularity. Such a regularity gap is sometimes required for utilizing efficient solvers like semismooth Newton methods.

\textbf{b. Conforming discretization}. For each $n\in\mathbb{N}$, let $X_n$ be a finite dimensional subspace of $X$ such that $X_n\subseteq X_{n+1}$ and that the $X$ closure of $Y:=\cup_{n\in\mathbb{N}}X_n$ equals $X$. Then, we set
\begin{equation*}
G_n(w):=\mathcal{I}_{X_n}(w),
\end{equation*}
the indicator function of $X_n$, i.e., $G_n(w)=0$ if $w\in X_n$ and $+\infty$ otherwise. Observe that this is a common setting for conforming finite element methods when discretizing the original problem. For example, if $X=H^1(\Omega)$, then $X_n$ can be taken to be the space of globally continuous, piecewise affine functions on a given mesh.
Utilizing a hierarchical (geometric) mesh refinement one obtains nested spaces satisfying
$X_n\subset X_{n+1}$, $n\in\mathbb{N}$.

The relation of a sequence $(u_n)_{n\in\NN}$ where $u_n$ solves \eqref{eq:AproxP} to solutions of  \eqref{eq:OrigP} is highly related to the $\Gamma$-convergence of $F_n$ to $F$. In particular, if $\Gamma$--$\lim F_n=F$ in the strong (weak) topology, then strong (weak) accumulation points of $(u_n)_{n\in\NN}$ solve \eqref{eq:OrigP}. The $\Gamma$-convergence is in turn closely related to the following density property:
\begin{equation}\label{Density}
\overline{K\cap Y}^X=K.
\end{equation}
Specifically, if \eqref{Density} holds, then $\Gamma$--$\lim F_n=F$ in the weak and strong topologies. On the other hand, if \eqref{Density} does not hold, then for any $u^*\in K\setminus \overline{K\cap Y}^X$, there are subsequences $(\gamma_{k_{n}})_{n\in\NN}$ (or $(\gamma_{k_{n}})_{n\in\NN}, (\delta_{k_{n}})_{n\in\NN}$ when relating to our example of Tikhonov regularization) for which $F_{k_{n}}(u_{k_{n}})\not\to F(u^*)$, as $n\to\infty$,  for any sequence $(u_n)_{n\in\NN}$ such that $u_{k_{n}}\to u^*$.

We emphasize that density results of the type \eqref{Density} can not be inferred, in general, from the dense embedding $Y\hookrightarrow X$; see \cite{MR3306389} for a striking counterexample. In \cite{MR3306389, Simon1} such dense embeddings of convex intersections have been obtained, but only for non-vanishing bounds or under additional regularity assumptions on the bound. In this work we close several gaps in the literature. Most notably, in Section \ref{sec:density} we show that the variable step mollifier $T$ is the right tool to establish density results for possibly vanishing bounds $\alpha$.


\subsection*{Summary of the results and organisation of the paper}

The paper is organised as follows.

 In Section \ref{sec:setup}, we fix the pair $(\rho, \eta)$ explicitly,  establish basic results for the operator $T$, and introduce the approximations $T_n$. In particular we show that for a function $f\in L_{loc}^{1}(\om)$, $Tf$ is smooth at $x\in\om$, where $\eta(x)>0$. In this section, we also {provide a first insight} into the mollification effect by $T$, by considering its action on continuous functions. Among others, we show that if $f\in C(\overline{\om})$, then $Tf=f$ at the points where $\eta$ vanishes and $T_{n}f\to f$ uniformly.

In Section \ref{sec:approximationLp} we provide basic boundedness and approximation results for  $L^{p}$ functions, $1<p\le \infty$. With the help of the Marcinkiewicz interpolation theorem we show that the operator $T$ is bounded from $L^{p}(\om)$ to $L^{p}(\om)$ for every $1<p\le \infty$. We also show the corresponding approximation result, i.e., $T_{n}f\to f$ in $L^{p}(\om)$ for every $f\in L^{p}(\om)$, $1<p\le \infty$.
 
 Analogous boundedness and approximation results for $W^{1,p}$ functions, $1<p\le \infty$ are the subjects of Section \ref{sec:approximationW1p}, where also a formula for the weak gradient of $Tf$ is derived. We also show that the operator preserves {the trace of $W^{1,p}$ functions }in the appropriate $W^{1-1/p,p}(\partial \om)$ sense. We put particular emphasis on the case where the function $\eta$ vanishes on a set inside $\om$, proving that the function $Tf$ has still a global weak derivative and all the corresponding results holds in that case, too.
 
 The case $p=1$, which covers $L^{1}(\om)$, $W^{1,1}(\om)$ and $\bv(\om)$ spaces, is treated separately in Section \ref{Sec:L1} since it is rather special. Indeed, issues similar to the ones appearing for the Hardy-Littlewood maximal operator become relevant.  In fact, the boundedness of $T$ in $L^{1}(\om)$ -- and consequently in $W^{1,1}(\om)$ and $\bv(\om)$  -- is rather involved to prove. In order to do so, we must invoke additional assumptions on the function $\eta$ which however do not affect any of the other results of this work.
 
Finally,  in Section \ref{sec:density} we focus on applications of our theoretical results. In particular, in Theorem \ref{thm:density} we prove the main density results of the type \eqref{Density} for possible vanishing obstacle functions $\alpha$.

\section{The set up and basic definitions}\label{sec:setup}

As stated before, we consider that $\om\subseteq \RR^{N}$ is an open, bounded domain, and, in general, we do not assume any regularity of the boundary $\partial \om$, unless explicitly mentioned. In this section we define the function $\rho$ that is used throughout the paper, and the basic map $\eta$ used either directly or to construct other maps. The pair $(\rho,\eta)$ defines uniequivocally the map $T$.

\subsection{Construction of $\rho$ and $\eta$}


We denote by $\rho$ a function with the following properties:
\begin{enumerate}
\item $\rho\in C_{c}^{\infty}(\RR^{N})$, $0\le \rho \le1$ with $\rho(x)=0 \iff|x|\ge 1$;
\item $\rho$ is radially symmetric.
\item $\rho(x)\ge \rho(1/2)$ for every $x$ with $|x|<1/2$.
\end{enumerate}
Note that, by slighly abusing the notation, above we use $\rho(1/2)=\rho(x)$ for any $x$ with $|x|=1/2$. Notice that the properties $\mathrm{(i)}$--$\mathrm{(iii)}$ above are satisfied, for instance, by the standard mollifying function
\begin{equation*}
\rho(x)=\left\{
  \begin{array}{ll}
    e^{-\frac{1}{1-|x|^2}}, & \hbox{ $|x|\leq 1$,} \\
    0, & \hbox{ $|x|>1$.} 
  \end{array}
\right.
\end{equation*}
Note that we do not scale $\rho$ to have unitary integral, as it is usually done when defining mollifiers.

The following result by Whitney \cite{whitney1934analytic} is of central importance to this work, and for the sake of completeness we provide here a proof. It states the existence of a non-negative smooth function vanishing on arbitrary closed sets and provides a main tool for defining $\eta$.

\newtheorem{eta_existence}{Theorem}[section]
\begin{eta_existence}[Whitney]\label{lbl:eta_existence}
For any closed set $\Theta\subseteq \RR^{N}$ there exists a non-negative function $\eta\in C^{\infty}(\RR^{N})$ such that 
\begin{enumerate}
\item $\Theta=\eta^{-1}(\{0\})$.
\item All derivatives of $\eta$ vanish on $\Theta$.
\item $\eta(x)< \mathrm{dist}(x,\Theta)$ for every {$x\in \RR^{N}\setminus \Theta$}. 
\end{enumerate}
\end{eta_existence}
\begin{proof}
 Note first that the open set $\RR^{N}\setminus \Theta$ can be written as a countable union of open balls $B_{i}$ each of which is compactly supported in $\RR^{N}\setminus \Theta$, i.e.,
\[\RR^{N}\setminus \Theta=\bigcup_{i=1}^{\infty} B_{i},\quad \text{where } B_{i}\Subset \RR^{N}\setminus \Theta\quad \text{for all }i\in\NN.\]
For every $B_{i}$ we associate a non-negative smooth function $\eta_{i}$ such that $\eta_{i}>0$ on $B_{i}$  and zero on the boundary of $B_i$. After an appropriate scaling, we may also assume that all partial derivatives of $\eta_{i}$ {up to order $i$} are uniformly bounded by $\frac{1}{\sqrt{N}}2^{-(i+1)}$. Define now
\[\eta(x)=\sum_{i=1}^{\infty} \eta_{i}(x),\quad x\in \RR^{N}.\]
Denoting by $\partial^{n}$ any partial derivative of order $n\ge 0$, we have that 
\[\left | \sum_{i=1}^{\infty}\partial^{n} \eta_{i}(x) \right |\le \left |\sum_{i=1}^{n-1}\partial ^{n} \eta_{i}(x) \right| +\left |\sum_{i=n}^{\infty}\partial^{n} \eta_{i}(x) \right |\le  \left |\sum_{i=1}^{n-1}\partial^{n}\eta_{i}(x) \right|+\frac{1}{\sqrt{N}}\sum_{i=n}^{\infty} 2^{-(i+1)},\]
and thus from the Weierstrass M-test the above sum converges uniformly on $\RR^{N}\setminus \Theta$. Hence,  the function $\eta$ is smooth and for every $n\in\NN$ we have $\partial^{n}\eta(x)=\sum_{i=1}^{\infty} \partial^{n}\eta_{i}(x),$ for  $x\in \RR^{N}$. In particular, since every $\eta_{i}\in C_{c}^{\infty}(\RR^{N}\setminus \Theta)$ for every $\xi\in \Theta$ we have that $\partial^{n}\eta(\xi)=0$, $n\ge 0$. From construction it obviously follows that $\eta>0$ on $\RR^{N}\setminus \Theta$. We have thus shown $(i)$ and $(ii)$. 

In order to show $(iii)$ notice first that $|\nabla \eta (x)|\le 1/2$ for every $x\in\RR^{N}$. Indeed, that follows from the fact that for any $j=1,\ldots,N$ we have
\[\left |\partial_{j}\eta(x)\right |=\left | \sum_{i=1}^{\infty} \partial_{j} \eta_i(x) \right |\le \frac{1}{\sqrt{N}}\sum_{i=1}^{\infty} 2^{-(i+1)}=\frac{1}{2\sqrt{N}},\]
where we use $\partial_j:=\partial/\partial x_j$. Since $\Theta$ is closed, we have that for every {$x\in \RR^{N}\setminus \Theta$}, there exists $\xi\in \Theta$ such that $|x-\xi|=\mathrm{dist}(x,\Theta)$. By an application of the mean value theorem for $\eta$ we have that for some $z\in[x,\xi]$, i.e., on the line segment joining $x$ and $\xi$,
\begin{align*}
\eta(x)=|\eta(x)-\eta(\xi)|\le |\nabla \eta (z)||x-\xi|< \mathrm{dist}(x,\Theta),
\end{align*}
where we used the fact that $|\nabla \eta(z)|\le 1/2$. Hence, the proof is complete.
\end{proof}

We now fix a smooth function $\eta$ as in Theorem \ref{lbl:eta_existence}, where we take
\begin{equation*}
\Theta=\partial \om, \qquad\text{or}\qquad \Theta=\partial\om \cup \Delta,
\end{equation*}
for some set $\Delta$ in $\overline{\Omega}$ such that $\Theta$ is closed. While $\partial\Omega$ is used to enforce the domain of integration of $T$ within $\Omega$, the set $\Delta$ is useful when dealing with sets $K$ as in \eqref{Kalpha} where $\alpha$ vanishes within $\Omega$. 
 For the sake of brevity, we denote by $\sigma$ the distance function to the boundary of $\om$, i.e., for every $x\in\om$, 
\[\sigma(x)=\mathrm{dist}(x,\partial\Omega)=\min\left\{ |x-\xi|:\; \xi\in \partial \om\right\}.\]
 Since $\mathrm{dist}(x,\Theta)\le  \mathrm{dist}(x,\partial \om)$, from $(iii)$ of Theorem \ref{lbl:eta_existence} we get
\begin{equation}\label{eta_less_dist}
\eta(x)< \sigma(x),\quad \text{for every }x\in\om.
\end{equation}
Finally, we denote the support of a function $f:\Omega\to \mathbb{R}$ as $\mathrm{supp} (f)$, i.e.,
\begin{equation*}
\mathrm{supp} (f):=\overline{\{x\in\Omega:f(x)\neq0\}}.
\end{equation*}

%
%


\subsection{Basic properties of $T$ and $T_n$} We now provide some elementary properties of $T$ and the sequence $(T_n)_{n\in\NN}$ (where $\eta$ is replaced by $\eta/n$) which subsequently help to show that they approximate the identity in a particular sense. A few  remarks about Definition \ref{lbl:def_T} are in order. In particular note the following: If $\eta(x)=0$, then $Tf(x)=f(x)$, and if $\eta(x)>0$ then
\begin{align}\label{intro:T}
Tf(x)=C(x)\int_{\om} \rho \left (\frac{x-y}{\eta(x)} \right )f(y)\,dy,
\end{align}
with
\begin{equation}
C(x):=\left (\int_{\om} \rho \left (\frac{x-y}{\eta(x)} \right )\,dy \right )^{-1}=\frac{M_\rho}{\eta(x)^N}.
\end{equation}
Notice first that for every $x\in\om$ such that $\eta(x)>0$, the map $y\mapsto \rho((x-y)/\eta(x))$ is smooth in $\om$ and, in view of \eqref{eta_less_dist}, is compactly supported in $\om$. In fact the domain of integration in \eqref{intro:T} can be equivalently taken to be $B_{\eta(x)}(x)$ instead of all of $\om$. Hence the operator $T$ is indeed well defined on $L_{loc}^{1}(\om)$. It is also straightforward to see that $Tf$ is infinitely differentiable at points $x\in\om$ with $\eta(x)>0$, and that if $\Theta=\partial\Omega$ then  $Tf\in C^{\infty}(\om)$. In order to see for example that $Tf$ is infinitely differentiable at a point $x_{0}\in\om$ with $\eta(x_0)>0$, consider the map $x\mapsto \rho  ((x-y)/\eta(x) )$ restricted to a small ball $B_{\delta}(x_{0})\Subset \om$, with $y$ fixed. This map is smooth in $B_{\delta}(x_{0})$ with all its derivatives being bounded there. From property \eqref{eta_less_dist} we have that
\[D:=\bigcup_{x\in B_{\delta}(x_{0})} B_{\eta(x)}(x) \Subset \om,\]
and thus
\[F(x):=(C(x))^{-1}Tf(x)=\int_{D} \rho\left (\frac{x-y}{\eta(x)} \right )f(y)\,dy,\quad x\in B_{\delta}(x_{0}).\]
In order to perform differentiation under the integral on the right hand side in the definition of $F$, it suffices to check that the $k$-th order partial derivatives of the map $x\mapsto \rho \left (\frac{x-y}{\eta(x)} \right )$ are uniformly bounded in $y$ by a constant which depends only on $k$. This follows from the fact that $\rho$ and $\eta$ are smooth, $\eta$ is bounded away from zero in $B_{\delta}(x_{0})$, and from the fact that the $k$-th order partial derivatives of $x\mapsto \rho \left (\frac{x-y}{\eta(x)} \right )$ are sums of products of terms which are one of the following, respectively:
\begin{itemize}
\item Partial derivatives of $\rho$ up to order $k$, evaluated at $\frac{x-y}{\rho(x)}$.
\item Partial derivatives of $\eta$ up to order $k$ evaluated at $x$.
\item Terms of the type $\frac{1}{\eta(x)^{n}}$, $1<n\le k$.
\item Products of the type $\prod_{i=1}^{k} (x_{N_{i}}-y_{N_{i}})$.
\end{itemize}
We refer also to Appendix \ref{sec:app} for more explicit statements concerning the first-order derivatives.

We now define the following sequence of approximations of $f\in L_{loc}^{1}(\om)$:

\newtheorem{Definition_T}[eta_existence]{Definition}
\begin{Definition_T}
Let $f\in L_{loc}^{1}(\om)$. We define $T_nf(x)$ for $x\in\om$ and $n\in\mathbb{N}$, as
\begin{align}\label{lbl:def_Tn}
T_nf(x):=M_{\rho}\int_{B_1(0)} \rho (z ) f\left(x-\frac{\eta(x)}{n}z\right)\,dz, & & \text{with} & &M_{\rho}:=\left(\int_{B_1(0)}\rho(y)\,dy\right)^{-1}.
\end{align}
\end{Definition_T}

Note in particular that if $\eta(x)>0$, then
\begin{equation}\label{Tnf_def}
T_{n}f(x):=C_{n}(x) \int_{\om} \rho\left (\frac{x-y}{\frac{1}{n}\eta(x)} \right )f(y)\,dy,\quad \text{with}\quad C_{n}(x):=M_{\rho}\frac{n^{N}}{\eta(x)^{N}}, \quad \text{for all } n\in\NN.
\end{equation}

In order to have a first insight into the difference between the smoothing procedure of the operator $T$ and the one of a standard mollification, we provide the following proposition which states that the operator preserves the boundary values of a function in $C(\overline{\om})$.

\newtheorem{C_boundary_preservation}[eta_existence]{Proposition}
\begin{C_boundary_preservation}\label{lbl:C_boundary_preservation}
If $f\in C(\overline{\om})$, then  $Tf\in C(\overline{\om})$, $Tf=f$ in $\partial \om \cup \Delta$, and $T_{n}f\to f$ uniformly as $n\to\infty$. Further, if $f\in C(\Omega)$ and $\Gamma\subseteq \partial\Omega$ is a closed set such that $\mathrm{supp}(f)\cap \Gamma=\emptyset$, then also $\mathrm{supp}(Tf)\cap \Gamma=\emptyset$.
\end{C_boundary_preservation}

\begin{proof}
Let $f\in C(\overline{\Omega})$, $\xi\in \Theta=\partial\om\cup \Delta$, $(x_{n})_{n\in\NN}\subseteq \om\setminus \Theta$ with $x_{n}\to \xi$. We will show that $Tf(x_{n})\to f(\xi)$. Since $\eta(x_n)>0$, we have
\begin{align*}
|Tf(x_{n})-f(\xi)|&=\left |C(x_{n})\int_{\om} \rho\left (\frac{x_{n}-y}{\eta(x_{n})} \right )\left (f(y)-f(\xi) \right )\,dy \right |\\
&\le \sup_{y\in B_{\eta(x_{n})}(x_{n})} |f(y)-f(\xi)| \underbrace{C(x_{n})\int_{\om} \rho\left (\frac{x_{n}-y}{\eta(x_{n})} \right )\,dy}_{=1}\\
&=\sup_{y\in B_{\eta(x_{n})}(x_{n})} |f(y)-f(\xi)|\\
&\le |f(x_{n})-f(\xi)|+\sup_{y\in B_{\eta(x_{n})}(x_{n})} |f(y)-f(x_{n})|\to 0,
\end{align*}
where we used the fact that $\eta(x_{n})\to 0$ as well as the uniform continuity of $f$. Thus, this proves that $Tf\in C(\overline{\om})$ and $Tf=f$ in $\partial \om \cup \Delta$.

In order to show the uniform convergence, note that if $x\in \Theta$, then $T_nf(x)=f(x)$, and for every $x\in\om\setminus \Theta$  we have
\begin{align*}
\left | T_{n}f(x)-f(x)\right |
&\le C_{n}(x)\int_{\om} \rho\left (\frac{x-y}{\frac{1}{n}\eta(x)} \right )|f(y)-f(x)|\,dy\\
&\le \sup_{y\in B_{\frac{1}{n}\eta(x)}(x)} |f(y)-f(x)|\to 0 \quad \text{as }n\to \infty,
\end{align*}
where the last quantity is going to zero independently of $x$, due to uniform continuity of $f$.

 Let $f\in C(\Omega)$, $\mathrm{supp}(f)\cap \Gamma=\emptyset$ and $\Gamma\subseteq\partial\Omega$  where $\Gamma$ is closed.  Then, there exists $\epsilon>0$ such that if $d(x,\Gamma)<\epsilon$, then $f(x)=0$. Also, since $\eta$ is smooth, $\eta(\Gamma)=0$ and also all the derivatives of $\eta$ vanish on $\Gamma$,  there exists $\epsilon'>0$, such that if $d(x,\Gamma)<\epsilon'$ we have $B_{\eta(x)}(x)\subseteq \{x:d(x,\Gamma)<\epsilon\}$. Then, if $\eta(x)>0$, $Tf$ is equivalently defined as
\begin{align}
Tf(x)=C(x)\int_{B_{\eta(x)}(x)} \rho \left (\frac{x-y}{\eta(x)} \right )f(y)\,dy.
\end{align}
We further have that $Tf(x)=0$ if $d(x,\Gamma)<\epsilon'$. Also, if $\eta(x)=0$, then $Tf(x)=f(x)=0$ if $d(x,\Gamma)<\epsilon$, i.e., $\mathrm{supp}(Tf)\cap \Gamma=\emptyset$.
\end{proof}

In particular note that the above results imply  that the space $C_c(\Omega)$ is $T$-invariant, i.e., $T(C_c(\Omega))\subseteq C_c(\Omega)$. In fact, $f\in C_c(\Omega)$ if and only if $f\in C(\Omega)$ and there exists an $\epsilon>0$ such that if $d(x,\partial \Omega)<\epsilon$ we have $f(x)=0$.

%

\section{Approximation results in $L^{p}$ spaces}\label{sec:approximationLp}

It will be convenient for this section as well as for the following ones  to denote by $\mathscr{L}(X)$ the space of bounded, linear operators from $X$ to $X$, where $(X, \|\cdot\|{\color{red}_{X}})$ is a Banach space.

We are now interested in investigating how the operator $T$ acts on $L^{p}(\om)$ functions. Throughout this section we consider the case $1<p\le \infty$. The case $p=1$ has some extra complications and is treated separately in Section \ref{Sec:L1} below, under additional assumptions on $\eta$. 
Our aim is to show that the operator $T$ belongs to $\mathscr{L}(L^{p}(\om))$ for $1<p\le \infty$. In order to do so, we  employ the Marcinkiewicz interpolation theorem; see for instance \cite[Chapter VIII, Theorem 9.1]{dibenedetto2002real}.

For this matter, we will show that $T$ belongs to $\mathscr{L}(L^{\infty}(\om))$ and it is also bounded from $L^{1}(\om)$ to $L^{1,w}(\om)$, i.e.,
\begin{align}
\|Tf\|_{L^{\infty}(\om)}\le \|f\|_{L^{\infty}(\om)},\quad \text{for every }f\in L^{\infty}(\om),\label{interp_Linf}\\
\left | \left \{x\in \Omega:\; |Tf(x)|>\lambda \right \} \right |\le  \frac{M_{1}}{\lambda}\|f\|_{L^{1}(\om)},\quad \text{for every }f\in L^{1}(\om),\label{interp_wL1}
\end{align}
for an appropriate constant $M_{1}$ that does not depend on $f$.  In particular, the formulation of the  Marcinkiewicz interpolation theorem associated to the above inequalities entails the following:

\newtheorem*{Marc}{Theorem}
\begin{Marc}[Marcinkiewicz]\label{marc}
	Let $T$ be a linear map defined on $L^1(\Omega)$ and $L^\infty(\Omega)$, and suppose that it satisfies \eqref{interp_Linf} and \eqref{interp_wL1}. Then, $T$ is a bounded linear operator in $L^p(\Omega)$ for any $1<p<\infty$, and 
	\begin{equation*}
		\|Tf\|_{L^{p}(\om)}\le C\|f\|_{L^{p}(\om)},\quad \text{for every }f\in L^{p}(\om),
	\end{equation*}
	where $C$ depends only  on $M_1$ and $p$.
\end{Marc}

Note however that the estimate \eqref{interp_Linf} is trivial, and for \eqref{interp_wL1} we use an argument similar to the one utilized for singular integral operators. Further, note that if $\Delta\neq \emptyset$ in $\Theta=\partial\Omega\cup \Delta$, then \eqref{interp_wL1} is inferred if the following holds true
\begin{align}\label{interp_wL1-2}
\left | \left \{x\in \Omega\setminus \Delta:\; |Tf(x)|>\lambda \right \} \right |\le  \frac{\tilde{M}_{1}}{\lambda}\|f\|_{L^{1}( \Omega\setminus \Delta)},\quad \text{for every }f\in L^{1}(\Omega\setminus \Delta).
\end{align}
In order to see this, let $\Delta_\lambda=\{x\in \Omega\setminus \Delta:\; |Tf(x)|>\lambda \}$ and suppose that \eqref{interp_wL1-2} holds true, and $f\in L^1(\Omega)$. Then, we have
\begin{align*}
\left | \left \{x\in \Omega:\; |Tf(x)|>\lambda \right \} \right |&=  | \Delta_\lambda | +\left | \left \{x\in \Delta:\; |Tf(x)|>\lambda \right \} \right |\\
&=  | \Delta_\lambda | +\left | \left \{x\in \Delta:\; |f(x)|>\lambda \right \} \right |\\
&\leq   \frac{\tilde{M}_{1}}{\lambda}\|f\|_{L^{1}(\om\setminus \Delta)} +\frac{1}{\lambda}\|f\|_{L^{1}( \Delta)}\\
&\leq   \frac{\max(\tilde{M}_{1},1)}{\lambda}\|f\|_{L^{1}(\om)},
\end{align*}
where we have used that $Tf|_{\Delta}=f$. Hence, it is enough to consider $\Delta=\emptyset$ so that $\eta>0$ on $\Omega$.

\newtheorem{L1_weak}[eta_existence]{Proposition}
\begin{L1_weak}\label{lbl:L1_weak}
The operator $T$ is bounded from $L^{1}(\om)$ to $L^{1,w}(\om)$, i.e., \eqref{interp_wL1} holds.
\end{L1_weak}

\begin{proof}
We have to show \eqref{interp_wL1} for a suitable constant $M$.
Note initially that our assumption on the mollifier $\rho$ yield the bounds
\[\mathcal{X}_{B_{\eta(x)}(x)}(y)\ge \rho\left (\frac{x-y}{\eta(x)} \right )\ge \rho\left (1/2\right) \mathcal{X}_{B_{\frac{\eta(x)}{2}}(x)}(y),\quad \text{for every }x,y\in\om,\]
where $\mathcal{X}_S$ denotes the characteristic function of the set $S$. Note that in view of the inequality $|\{x\in\om:\,|Tf(x)|>\lambda \}|\le |\{x\in\om:\,T|f(x)|>\lambda \}|$ it suffices to prove \eqref{interp_wL1} for $f\ge 0$.
For $f\ge 0$ we get
\[\int_{\om}\mathcal{X}_{B_{\eta(x)}(x)}(y)f(y)\, dy\ge \int_{\om} \rho\left (\frac{x-y}{\eta(x)} \right ) f(y)\,dy\ge \rho\left (1/2\right)\int_{\om} \mathcal{X}_{B_{\frac{\eta(x)}{2}}(x)}(y) f(y)\,dy. \]
Consequently, 
\begin{align*}
\frac{M_{\rho}}{\eta(x)^{N}} \int_{\om} \mathcal{X}_{B_{\eta(x)}(x)}(y)f(y)\, dy
&\ge C(x)  \int_{\om} \rho\left (\frac{x-y}{\eta(x)} \right ) f(y)\,dy\ge \frac{ \rho\left (1/2\right)M_{\rho}}{\eta(x)^{N}}  \int_{\om} \mathcal{X}_{B_{\frac{\eta(x)}{2}}(x)}(y) f(y)\,dy.
\end{align*}
Define now the following sets
\begin{align*}
E:= \left \{ x:\; |Tf(x)|>\lambda \right \},\quad\text{and}\quad \tilde{E}:= \left \{x:\; \frac{M_{\rho}}{\eta(x)^{N}} \int_{\om} \mathcal{X}_{B_{\eta(x)}(x)}(y)f(y)\, dy> \lambda \right \},
\end{align*}
where $E\subseteq \tilde{E}$. Further, $\tilde{E}\subseteq \bigcup_{x\in \tilde{E}} B_{\eta(x)}(x)$ and thus from the Vitali covering lemma we have that
\[\tilde{E}\subseteq \bigcup_{x\in I} 5 B_{\eta(x)}(x),\]
where $I$ is a countable subset of $\tilde{E}$ and $B_{\eta(x_{i})}(x_{i})\cap B_{\eta(x_{j})}(x_{j})=\emptyset$ for every $x_{i},x_{j}\in I$ with $x_{i}\ne x_{j}$.
Finally, denoting by $\omega_{N}$ the volume of the unit ball in $\RR^{N}$, we get
\begin{align*}
\left | \left \{x:\; |Tf(x)|>\lambda \right \} \right |=|E|\le |\tilde{E}|
     &\le 5^{N} \sum_{x\in  I} |B_{\eta(x)}(x)|\\
     &= 5^{N} \omega_{N} \sum_{x\in I } \eta(x)^{N}\\
     &\le \frac{5^{N}\omega_{N}M_{\rho}}{\lambda} \sum_{x\in I} \int_{\om}  \mathcal{X}_{B_{\eta(x)}(x)}(y)f(y)\, dy\\
     &\le \frac{5^{N}\omega_{N} M_{\rho}}{\lambda} \int_{\om} |f(y)|\,dy
     = \frac{5^{N} \omega_{N}M_{\rho}}{\lambda} \|f\|_{L^{1}(\om)}.
\end{align*}
\end{proof}


The result concerning boundedness of $T$ in $L^p(\Omega)$ and the approximation property that the sequence of operators $(T_n)_{n\in\NN}$ strongly tend to the identity is stated next.

\newtheorem{convergence_Lp}[eta_existence]{Proposition}
\begin{convergence_Lp}\label{lbl:convergence_Lp}
The operator $T$ belongs to $\mathscr{L}(L^{p}(\om))$ for $1<p\le \infty$, and if  $f\in L^{p}(\om)$, $1<p< \infty$, then $T_{n}f\to f$ in $L^{p}(\om)$ as $n\to\infty$.
\end{convergence_Lp}

\begin{proof}
The proof of the $L^p$ boundedness follows directly from estimate \eqref{interp_Linf}, Proposition \ref{lbl:L1_weak} and  the application of Marcinkiewicz theorem. Let now $f\in L^{p}(\om)$ with $1<p<\infty$. We can find a function $\phi\in C_{c}(\om)$ such that $\|f-\phi\|_{L^{p}(\om)}\le \frac{\epsilon}{3}$.
Since each $T_{n}$ belongs to $\mathscr{L}(L^{p}(\om))$ there exists a constant $C$  such that 
\begin{equation}\label{Tf_Tphi}
\|T_{n}f-T_{n}\phi\|_{L^{p}(\om)}\le C\|f-\phi\|_{L^{p}(\om)}\le C\frac{\epsilon}{3}.
\end{equation}
The independence of $C$ of $n$ follows from the fact that for $T_n$, $M_1=5^N\omega_N M_\rho$ in \eqref{interp_wL1} is independent of $n$ (see the construction in the previous proposition), the bound in \eqref{interp_Linf} is identically $1$, and hence the operator norm of $T_n$ obtained via  the Marcinkiewicz interpolation theorem is independent of $n$. 

Moreover since $\phi\in C_{c}(\om)$, from Proposition \ref{lbl:C_boundary_preservation} we get that for large enough $n$
\begin{equation}\label{Tphi_phi}
\|T_{n}\phi-\phi\|_{L^{p}(\om)} \le M \frac{\epsilon}{3},
\end{equation}
for another constant $M>0$. Thus, finally for sufficiently large $n$ we have
\begin{align*}
\|T_{n}f-f\|_{L^{p}(\om)}&\le \|T_{n}f-T_{n}\phi\|_{L^{p}(\om)}+\|T_{n}\phi-\phi\|_{L^{p}(\om)}+\|\phi-f\|_{L^{p}(\om)}\\
					& \le C\frac{\epsilon}{3}+M\frac{\epsilon}{3}+\frac{\epsilon}{3}\\
					& \le \tilde{C}\epsilon,
\end{align*}
for some constant $\tilde{C}>0$.

\end{proof}

\section{Properties of $T$ on $W^{1,p}$ spaces}\label{sec:approximationW1p}

We now turn our attention to properties of $Tf$ for $f$ belonging to a suitable Sobolev space.
First, let $\Theta=\partial \Omega$.  Our goal is to establish boundedness of the operator in $W^{1,p}(\Omega)$ (and closed subspaces of interest) for $1<p\le \infty$. Subsequently, we treat the case  $\Theta=\Delta\cup \partial \Omega$ with $\Delta\ne\emptyset$. Note that the case $p=1$ is considered in Section \ref{Sec:L1} as it requires an alternative approach, and slightly different properties for $\eta$ are needed.

We emphasize that the case $\Delta\neq\emptyset $ in the present setting is significantly more complex than in the situation with Lebesgue spaces, as the function $Tf$ is defined by
\[
Tf(x)=
\begin{cases}
   C(x)\int_{\om} \rho \left (\frac{x-y}{\eta(x)} \right )f(y)\,dy,  & \;x\in \Omega\setminus \Delta, \\
    f(x), & \; x\in \Delta.
  \end{cases}
\]
Hence, it is necessary to prove that the above function has a global weak derivative. We start by showing that $T$ is bounded in certain Sobolev spaces for $\Theta=\partial\Omega$, and  then transfer this result via a limiting process to $\Theta=\partial\Omega\cup \Delta$ with $\Delta\neq \emptyset$. 

In addition to the usual $W_{0}^{1,p}(\Omega)$ and $W^{1,p}(\Omega)$ spaces \cite{adams2003sobolev}, we define the Banach space $W^{1,p}_\Gamma(\Omega)$ for $1<p<\infty$ by
\begin{equation*}
W^{1,p}_\Gamma(\Omega):=\overline{\{w\in C^\infty(\Omega): \mathrm{supp} (w)\cap \Gamma=\emptyset\}}^{W^{1,p}(\Omega)},
\end{equation*}
where $\Gamma\subseteq \partial \Omega$ is a closed set.  Note that the spaces $W^{1,p}_\Gamma(\Omega)$ are of utmost importance in the treatment of partial differential equations (PDEs) with mixed (Dirichlet/Neumann/Robin) boundary conditions. It follows that  $W^{1,p}_\Gamma(\Omega)=W^{1,p}(\Omega)$ if $
\Gamma=\emptyset$, and $W^{1,p}_\Gamma(\Omega)=W^{1,p}_0(\Omega)$ if $
\Gamma=\partial\Omega$.

\subsection{The case $\boldsymbol{\Delta= \emptyset}$}
In the next proposition, we compute the gradient of $Tf$, where for the sake of full generality we assume that $f\in \bv_{loc}(\om)$. Recall that the space $\bv_{loc}(\om)$ consists of all the functions $f\in L_{loc}^{1}(\om)$  such that $V(u,A):=\sup\{\int_{\om} u\,\mathrm{div} \phi\,dx:\; \phi\in C_{c}^{\infty}(A,\RR^{N}),\;\|\phi\|_{\infty}\le 1 \}<\infty$ for every open set $A$ which is {compactly supported} inside $\om$.  In that case there exists a $\RR^{N}$-valued Radon measure, denoted here by $Df\in\mathcal{M}_{loc}(\om)$, that represents the distributional derivative of $f$. If $f\in L^{1}(\om)$ and  $V(u,\om)<\infty$ then $Du$ is a finite Radon measure and we say that $f$ is a function of bounded variation. We then write $f\in \bv(\om)$. See \cite{AmbrosioBV} for more details.

In our subsequent exposition, we denote the gradient as a row vector.


\newtheorem{Tf_derivative}[eta_existence]{Proposition}
\begin{Tf_derivative}\label{lbl:Tf_derivative}
Let $f\in \bv_{loc}(\om)$ and $\Theta=\partial \Omega$. Then the gradient of $Tf\in C^{\infty}(\om)$ is given by
\begin{equation}\label{Tf_derivative_bv}
\nabla Tf(x)=C(x)\int_{\om} \rho \left ( \frac{x-y}{\eta(x)} \right )\, dDf(y)+C(x)\frac{\nabla \eta(x)}{\eta(x)} \int_{\om} \rho \left ( \frac{x-y}{\eta(x)} \right ) (y-x)^{\top}\cdot dDf(y).
\end{equation}
\end{Tf_derivative}

\begin{proof}
The proof is based on basic differentiation rules for $\bv$ functions and can be found in Appendix \ref{sec:app}.
\end{proof}

Note that if $f\in W_{loc}^{1,1}(\om)$ then \eqref{Tf_derivative_bv} can be written as
\begin{equation}\label{Tf_derivative_w11}
\nabla Tf(x)= T(\nabla f)(x)+ C(x)\frac{\nabla \eta(x)}{\eta(x)}\int_{\om} \rho \left (\frac{x-y}{\eta(x)} \right )( y-x )^{\top}\nabla f(y)\, dy.
\end{equation}
Here $T(\nabla f)$ is acting componentwise. Since $|x-y|\le \eta(x)$ in \eqref{Tf_derivative_w11}, notice that we immediately obtain the following pointwise estimate for $f\in W_{loc}^{1,1}(\om)$:
\begin{equation}\label{bound_derivative}
|\nabla Tf(x)|\le |T(\nabla f)(x)|+ |\nabla \eta(x)| T(|\nabla f|)(x), \quad x\in\om.
\end{equation}
Bearing this in mind the following proposition follows immediately.
\newtheorem{W1p_bound}[eta_existence]{Proposition}
\begin{W1p_bound}\label{lbl:W1p_bound}
Let   $\Theta=\partial \Omega$. The operator $T$ belongs to $\mathscr{L}(W_\Gamma^{1,p}(\om))$ for $1<p< \infty$, and if $f\in W_\Gamma^{1,p}(\om)$, $1<p< \infty$, then $T_{n}f\to f$ in $W^{1,p}_\Gamma(\om)$ as $n\to\infty$.
\end{W1p_bound}
\begin{proof}
Let $f\in W_\Gamma^{1,p}(\om)$. Note that $\max_{x\in\om}|\nabla\eta(x) |<+\infty$ and since $T\in \mathscr{L}(L^p(\Omega))$, we have $|T(\nabla f)|, T(|\nabla f|)\in L^p(\Omega)$. Hence,  by \eqref{bound_derivative}, $\nabla (Tf)\in L^p(\om)$, and since $Tf\in L^p(\Omega)$, we infer $T:W_\Gamma^{1,p}(\Omega)\to W ^{1,p}(\Omega)$. The boundedness of $T$ follows directly from \eqref{bound_derivative}, $\max_{x\in\om}|\nabla\eta(x) |<+\infty$, and $T\in \mathscr{L}(L^p(\Omega))$. 

 Since $f\in W_\Gamma^{1,p}(\om)$, by definition, there exists $(f_n)_{n\in\NN}$ with $f_n\in C^\infty(\Omega)\cap W^{1,p}(\om)$, $\mathrm{supp}(f_n)\cap \Gamma=\emptyset$ and $f_n\to f$ in  $W^{1,p}(\om)$. By Proposition \ref{lbl:C_boundary_preservation}, one gets $\mathrm{supp}(Tf_n)\cap \Gamma=\emptyset$, and by properties of $T$ we have $Tf\in C^\infty(\Omega)$. 
We further have for a generic constant $M>0$ that
\begin{align*}
\|\nabla Tf_{n}-\nabla Tf\|_{L^p(\Omega)}^{p}&\le M\left ( \|T(\nabla (f_{n}-f))\|_{L^p(\Omega)}^{p}+ \left(\max_{x\in\om}|\nabla\eta(x)|^{p} \right)\|T(|\nabla (f_{n}-f)|)\|_{L^p(\Omega)}^{p}\right )\\
&\le M \|\nabla (f_{n}-f)\|_{L^p(\Omega)}^{p},
\end{align*}
and further $\|Tf_{n}- Tf\|_{L^p(\Omega)}\leq C \|f_{n}-f\|_{L^p(\Omega)}$. Hence, $Tf_n\to Tf$ in $W^{1,p}(\Omega)$, so that $Tf\in W_\Gamma^{1,p}(\Omega)$. Hence, $T:W_\Gamma^{1,p}(\Omega)\to W_\Gamma^{1,p}(\Omega)$, and as a result $T\in\mathscr{L}(W_\Gamma^{1,p}(\Omega))$.

From Proposition \ref{lbl:convergence_Lp} we get that $\|T_{n}f-f\|_{L^{p}(\om)}\to 0$ as $n\to\infty$. Thus, it suffices to show that $\|\nabla T_{n} f-\nabla f\|_{L^{p}(\om)}\to 0$ to prove that $T_{n}f\to f$ in $W^{1,p}_\Gamma(\om)$. Keeping in mind that  from Proposition \ref{lbl:convergence_Lp} we also get that $\|T_{n}(\nabla f)-\nabla f\|_{L^{p}(\om)}\to 0$, from the triangle inequality it is enough to show
\[\|\nabla T_{n} f-T_{n}(\nabla f)\|_{L^{p}(\om)}\to 0 \quad \text{as }n\to\infty.\]
According to Proposition \ref{lbl:Tf_derivative} and similarly to \eqref{bound_derivative} we have that 
\begin{equation}\label{bound_derivative2}
|\nabla T_nf(x)-T_n(\nabla f)(x)|\leq \frac{|\nabla \eta(x)|}{n} T_n(|\nabla f|)(x), 
\end{equation}
for  all $x\in \Omega$. Since $\|T_{n}(\nabla f)\|_{L^p(\Omega)}\to \|\nabla f\|_{L^p(\Omega)}$ we further find
\begin{align*}
\|\nabla T_{n} f-T_{n}(\nabla f)\|_{L^p(\Omega)}
\le \frac{\left(\max_{x\in\om}|\nabla\eta(x)| \right)}{n} \|T_{n}(|\nabla f|)\|_{L^p(\Omega)} \to 0\quad \text{as }n\to\infty,
\end{align*}
which completes the proof.
%
%
\end{proof}

As in the case of continuous functions, the operator $T$ preserves the traces of $W^{1,p}(\om)$ functions. This is shown in the following proposition. Note that in this case, some regularity for the boundary of $\om$ must be assumed in order to define the trace operator. 

\newtheorem{trace_W1p}[eta_existence]{Proposition}
\begin{trace_W1p}\label{lbl:trace_W1p}
Suppose that $\Theta=\partial \Omega$, $\om$ has Lipschitz boundary and $f\in W^{1,p}(\om)$, with $1<p< \infty$. Then
\begin{equation}\label{trace_W1p_eq}
\tau f=\tau(Tf),
\end{equation}
where $\tau$ denotes the zero-order boundary trace and the above equality is considered in $W^{1-1/p,p}(\partial \om)$. 
\end{trace_W1p}

\begin{proof}
 Let $f\in W^{1,p}(\om)$, with $1<p< \infty$, and consider a sequence $(f_{n})_{n\in\NN}\subseteq C^{\infty}(\overline{\om})$ convergent to $f$ in $W^{1,p}(\om)$. Note that for this sequence the equality \eqref{trace_W1p_eq} is satisfied. Indeed from Proposition \ref{lbl:C_boundary_preservation} we have that $Tf_{n}=f_{n}$ on $\partial\om$. Moreover, for a fixed $n\in\NN$, $\xi\in\partial \om$  and for $(x_{k})_{k\in\NN}$ with $x_{k}\to \xi$, using \eqref{Tf_derivative_w11} we get
\begin{align*}
|T\nabla f_{n}(x_{k}) - \nabla f_{n} (\xi)|& \le |T(\nabla f_{n})(x_{k})-\nabla f_{n}(\xi)|+|\nabla \eta(x_{k})| T(|\nabla f_{n}|)(x_{k}),
\end{align*}
with the righthand side going to zero as $k\to\infty$. This follows from Proposition \ref{lbl:C_boundary_preservation} yielding $T(\nabla f_{n})(x_{k})\to \nabla f_{n}(\xi)$, as $\nabla f_{n}\in [C(\overline{\om})]^{N}$, $|\nabla \eta(x_{k})|\to 0$ by construction, and since $T(|\nabla f_{n}|)$ is bounded. Thus $T\nabla f_{n}=\nabla f_{n}$ on $\partial \om$.

 Using Proposition \ref{lbl:W1p_bound} and the continuity of the trace with respect to the strong $W^{1,p}$-convergence, we have for a generic constant $M>0$
\begin{align*}
\|\tau f-\tau(Tf)\|_{W^{1-1/p,p}(\partial \om)}
&\le \| \tau f- \tau f_{n}\|_{W^{1-1/p,p}(\partial \om)}+\| \tau f_{n}-\tau(Tf)\|_{W^{1-1/p,p}(\partial \om)}\\
&\le \|\tau f-\tau f_{n}\|_{W^{1-1/p,p}(\partial \om)}+\|\tau(Tf_{n})-\tau(Tf)\|_{W^{1-1/p,p}(\partial \om)}\\
&\le M \|f-f_{n}\|_{W^{1,p}(\om)}+ M\|Tf_{n}-Tf\|_{W^{1,p}(\om)}\\
&\le M \|f-f_{n}\|_{W^{1,p}(\om)}\to 0.
\end{align*}
Hence, \eqref{trace_W1p_eq} holds, and the proof is complete.
\end{proof}

\subsection{The case $\boldsymbol{\Delta\neq \emptyset}$}

\newtheorem{theorem}{Theorem}[section]
We are now in shape to consider the operator $T$ defined with a  function $\eta$ vanishing on any set $\Delta$ contained in $\Omega$ -- such that $\Delta\cup \partial \om$ is closed -- in addition to also vanishing on $\partial\Omega$. For this purpose, consider $\eta^0$ and $\eta^1$ to be the functions obtained via Whitney's theorem for $\Theta=\partial\Omega$ and $\Theta=\partial\Omega\cup \Delta$, respectively, and suppose that $\eta^{i}(x)\leq \sigma(x)/2$ for $x\in \Omega$, $i=1,2$. 
We denote by $T$  the variable step mollifier associated to $\eta^1$, and now we define the limit and a sequence of variable step mollifiers $T^n$.

Let $f\in L_{loc}^1(\Omega)$, we define $\tilde{T}f:\Omega\to \mathbb{R}$ as
\begin{equation}\label{Tf_def_general}
\tilde{T}f(x):=\lim_{n\to\infty} (T^nf)(x), 
\end{equation}
for $x\in \Omega$ where $T^n$ is defined as
\begin{equation*}
T^nf(x):=C^n(x)\int_{\RR^{N}} \rho\left (\frac{x-y}{\eta^1(x)+\frac{1}{n}\eta^0(x)} \right ) f(y)\,dy, \quad\text{with}\quad C^n(x):=\frac{M_\rho}{(\eta^1(x)+\frac{1}{n}\eta^0(x))^N}.
\end{equation*}
Indeed, $T^n$ is the mollifier associated to $\eta=\eta^1+\frac{1}{n}\eta^0$. The following results also hold in this general case.

\newtheorem{W1p_Delta_nonempty}[eta_existence]{Proposition}
\begin{W1p_Delta_nonempty}\label{lbl:W1p_Delta_nonempty}
Let $f\in L_{loc}^1(\Omega)$, then the limit in \eqref{Tf_def_general} is well-defined and $\tilde{T}f=Tf$, where this equality holds almost everywhere in $\Delta$ and everywhere in $\om\setminus \Delta$.
Further, $T$ belongs to $\mathscr{L}(W_\Gamma^{1,p}(\Omega))$ for $1<p<\infty$.
\end{W1p_Delta_nonempty}
\begin{proof}
Let $x\in \Delta $, then we have that $\eta^1(x)=0$ and hence
\begin{align*}
|T^n(f)(x)-f(x)|&\leq \frac{1}{\int_{\RR^{N}} \rho\left (\frac{x-y}{\frac{1}{n}\eta^0(x)} \right ) \,dy}\int_{\RR^{N}} \rho\left (\frac{x-y}{\frac{1}{n}\eta^0(x)} \right ) |f(y)-f(x)|\,dy \\
&\leq M_{\rho}\frac{ n^{N}}{ \eta^{0}(x)^{N}} \int_{\om} \mathcal{X}_{B_{\eta^0(x)/n}(x)}(y)|f(y)-f(x)|\, dy.
\end{align*}
Since $f\in L_{loc}^1(\Omega)$, almost every point in $x\in \Omega$ is a Lebesgue point. Thus, the above limit as $n\to\infty$ is zero almost everywhere in $\Omega$, which proves $\tilde{T}f=f$ almost everywhere in $\Delta$. The fact that $\tilde{T}f=Tf$ everywhere in $\Omega\setminus \Delta$ follows by application of the dominated convergence theorem.  Further, analogously as for $T_n$, it follows that $(\|T^n\|_{\mathscr{L}(L^p(\Omega))})_{n\in\mathbb{N}}$ is uniformly bounded.
  By the dominated convergence theorem we have that if $\phi\in C(\overline{\om})$, then $T^{n}\phi \to T\phi$ in $L^{p}(\om)$. Proceeding as in Proposition \ref{lbl:convergence_Lp}, given $\epsilon>0$ we can find a function $\phi\in C_{c}(\om)$ such that $\|f-\phi\|_{L^{p}(\om)}\le \epsilon/3$. Hence, for sufficiently large $n$, and for a generic constant $M$ we have
\begin{align*}
 \|T^{n}f-Tf\|_{L^{p}(\om)}
 & \le \|T^{n}f - T^{n}\phi\|_{L^{p}(\om)} + \|T^{n}\phi-T\phi\|_{L^{p}(\om)}+ \|T\phi-Tf\|_{L^{p}(\om)}\\
 & \le  M \|f-\phi\|_{L^{p}(\om)} +\frac{\epsilon}{3} + M \|\phi-f\|_{L^{p}(\om)}\\
 &\le M\epsilon.
\end{align*}
  This shows that $T^{n}f\to Tf$ in $L^{p}(\om)$ for every $f\in L^{p}(\om)$.


Let now $f\in W^{1,p}(\Omega)$, then $T^nf\in W^{1,p}(\Omega)$ and $\nabla T^nf$ is given by
\begin{align*}
\nabla T^nf(x)&= T^n(\nabla f)(x)+\psi_n(x),
\end{align*}
where
\begin{equation*}
\psi_n(x):=C^n(x)\frac{\nabla \eta^1(x)+\frac{1}{n}\nabla\eta^0(x)}{\eta^1(x)+\frac{1}{n}\eta^0(x)}\int_{B_{\eta^1(x)+\frac{1}{n}\eta^0(x)}(x)} \rho \left (\frac{x-y}{\eta^1(x)+\frac{1}{n}\eta^0(x)} \right )( y-x )^{\top}\nabla f(y)\, dy.
\end{equation*}
It follows that $T^n(\nabla f)\to T(\nabla f)$ in $ L^p(\Omega)$  by the first paragraph of the proof. Now we prove that $\psi_n\to \psi$ in $ L^p(\Omega)$ with
\begin{equation*}
\psi(x):=
\left\{
  \begin{array}{ll}
    C(x)\frac{\nabla \eta^1(x)}{\eta^1(x)}\int_{B_{\eta^1(x)}(x)} \rho \left (\frac{x-y}{\eta^1(x)} \right )( y-x )^{\top}\nabla f(y)\, dy, & \hbox{ $x\in \Omega\setminus \Delta$,} \\
    0, & \hbox{ $x\in \Delta$.}
  \end{array}
\right.
\end{equation*}
Let $x\in \Delta$, so that $\eta^1(x)=0$, and
\begin{align*}
|\psi_n(x)|&\leq C^n(x)\frac{\nabla\eta^0(x)}{\eta^0(x)}\int_{B_{\frac{1}{n}\eta^0(x)}(x)} \rho \left (\frac{x-y}{\frac{1}{n}\eta^0(x)} \right )| y-x | |\nabla f(y)|\, dy\\
&\leq \frac{\nabla\eta^0(x)}{n} C^n(x)\int_{B_{\frac{1}{n}\eta^0(x)}(x)} \rho \left (\frac{x-y}{\frac{1}{n}\eta^0(x)} \right )|\nabla f(y)|\, dy.
\end{align*}
Since $\nabla\eta^0$ is bounded, and 
\begin{equation*}
C^n(\cdot)\int_{B_{\frac{1}{n}\eta^0(\cdot)}(\cdot)} \rho \left (\frac{\cdot-y}{\frac{1}{n}\eta^0(\cdot)} \right )|\nabla f(y)|\, dy\to |\nabla f|\quad \text{in} \quad L^p(\Omega),
\end{equation*}
it follows that $\psi_n \chi_{\Delta}\to 0$ in $L^p(\Omega)$.  If $x\in \Omega\setminus \Delta$, then application of the dominated convergence shows that $\psi_n(x)\to \psi(x)$, and analogously as done in Section \ref{sec:approximationLp}, it is shown that $\psi_n \chi_{\Omega\setminus \Delta}\to \psi \chi_{\Omega\setminus \Delta}$ in $L^p(\Omega)$. Therefore,
\begin{align*}
\int_{\Omega} Tf(x)\mathrm{div }{\phi}(x)\mathrm{d}x &=\lim_{n\to\infty} \int_{\Omega} T^nf(x)\mathrm{div }{\phi}(x)\mathrm{d}x =-\lim_{n\to\infty} \int_{\Omega} \nabla T^nf(x)\cdot {\phi}(x)\mathrm{d}x\\
&=-\int_{\Omega} (T(\nabla  f)+\psi)(x)\cdot {\phi}(x)\mathrm{d}x,
\end{align*}
for all $\phi\in C_c^\infty(\Omega)^N$. Therefore, $Tf$ is weakly differentiable with weak gradient $T(\nabla  f)+\psi$ that belongs to $L^p(\Omega)$, i.e., $Tf\in W^{1,p}(\Omega)$. Since the inequality \eqref{bound_derivative} holds here as well, in general in an almost everywhere sense, the boundedness of $T$ in $W^{1,p}(\om)$ follows again from the fact that $\max_{x\in\om} |\nabla \eta(x)|<\infty$ and $T\in\mathscr{L}(L^{p}(\om))$. Hence $T\in \mathscr{L}(W^{1,p}(\om))$ and it remains to show that $\mathscr{L}(W_{\Gamma}^{1,p}(\om))$. 

 If $f\in W_\Gamma^{1,p}(\Omega)$, then there exists $(f_n)_{n\in\NN}$ in $C^\infty(\Omega)$ such that $\mathrm{supp} (f_n) \cap \Gamma=\emptyset$, and $f_n\to f$ in $W^{1,p}(\Omega)$. In order to show that $Tf\in W_{\Gamma}^{1,p}(\om)$ as well, it suffices to show that for  every $k\in\NN$, there exists a $g_{k}\in W^{1,p}(\om)\cap C^{\infty}(\om)$ with $\supp(g_{k})\cap \Gamma=\emptyset$ such that
 \begin{equation}\label{1kover2}
\|Tf-g_{k}\|_{W^{1,p}(\om)}\le \frac{1}{k}.
\end{equation}
We fix $k\in\NN$. Since $f_{n}\to f$ in $W^{1,p}(\om)$ and $T\in \mathscr{L}(W_{\Gamma}^{1,p}(\om))$ we have that there exists $n_{0}\in\NN$ with
\begin{equation}\label{1kover2_a}
\|Tf_{n_{0}}-Tf\|_{W^{1,p}(\om)}\le \frac{1}{2k}.
\end{equation}
According to what has been shown above, we have that $T^{n}f_{n_{0}}\to Tf_{n_{0}}$ in $W^{1,p}(\om)$. Thus, we may choose $k\in\NN$ such that 
\begin{equation}\label{1kover2_b}
\|T^{k}f_{n_{0}}-Tf_{n_{0}}\|_{W^{1,p}(\om)}\le \frac{1}{2k}.
\end{equation}
 Setting $g_{k}:=T^{k}f_{n_{0}}$, from \eqref{1kover2_a} and \eqref{1kover2_b} we obtain \eqref{1kover2}. Noticing that $g_{k}\in  W^{1,p}(\om)\cap C^{\infty}(\om)$ and from Proposition \ref{lbl:C_boundary_preservation} we have that  $\supp(g_{k})\cap \Gamma=\emptyset$. This completes the proof. 
%

\end{proof}

\section{The cases of $L^{1}(\om)$, $W^{1,1}(\om)$ and $\bv(\om)$}\label{Sec:L1}

\subsection{$\boldsymbol{L^{1}(\om)}$ boundedness} \label{sec:L1_bound}
 We would like to show that $T$ also defines an operator in  $\mathscr{L}(L^{1}(\om))$. This is more involved than the cases with $p>1$. Consequently, we treat it separately here. It turns out that we require a lower bound on the decay of $\eta$ as a technical assumption. In order to see that the boundedness of the operator in $L^1(\om)$ is more delicate, consider the following counterexample.
 
 
Let $\Omega=(0,1)$,  $\rho=\chi_{(-1,1)}$ and $\eta=\mathrm{dist}(x,\partial \Omega)$, and the function $f_0(x)=(x\log (2/x))^{-2}$ so that $f_0\in L^1(0,1)$. However, $Tf_0(x)=(x\log (1/x))^{-1}$ for $x\in (0,1/2)$, so that $Tf_0\notin L^1(0,1)$. Even further, the same can be  obtained for a $\rho$ such that $\rho|_{\partial B_1(0)}=0$: Let $\rho$ be defined as $\rho(r)=g(r)$ if $r\geq 0$, and $\rho(r)=g(-r)$ if $r<0$, where $g:[0,1]\to\mathbb{R}$ is non-negative (non identically zero), decreasing, $g|_{(0,1)}\in C^\infty(0,1)$, $g(1)=0$, and such that
\begin{equation*}
\lim_{a\downarrow 0} g(a)\int_{a}^{1/2}\int_{-1+a}^{1+a}f_0(x(1-z))\,dz\,dx=+\infty.
\end{equation*}
Then, again $Tf_0\notin L^1(0,1)$, and $T$ is also a mollifier in the sense that if $f\in L^1_{loc}(\Omega)$, then $Tf\in C(\Omega)$.

 We develop the results in this section assuming that $\Delta=\emptyset$ so that $\Theta=\partial\Omega$. We emphasize, however, that all results can be extended to the case $\Delta\neq\emptyset$, as well. 
 
 We first state a necessary and sufficient condition for the operator to be bounded in $L^{1}(\om)$:
 
 \newtheorem{L1_L1}[eta_existence]{Proposition}
\begin{L1_L1}\label{lbl:L1_L1}
The operator $T$  belongs to $\mathscr{L}(L^{1}(\om))$ if and only if
\begin{equation}\label{L1_nece_suff}
\mathcal{K}:=\sup_{y\in\om}\int_{\om} C(x) \rho \left (\frac{x-y}{\eta(x)} \right )dx<\infty.
\end{equation}
\end{L1_L1}

\begin{proof}
Suppose \eqref{L1_nece_suff} holds and let $f\in L^{1}(\om)$. Then we have by virtue of Tonelli's theorem
\begin{align*}
\int_{\om}|Tf(x)|\,dx
&= \int_{\om} C(x)\left | \int_{\om}\rho \left (\frac{x-y}{\eta(x)} \right ) f(y) \,dy  \right | \,dx\\
&\le \int_{\om} |f(y)| \left (\int_{\om} C(x) \rho \left (\frac{x-y}{\eta(x)} \right )dx \right ) \,dy\\
&\le  \mathcal{K} \|f\|_{L^{1}(\om)}.
\end{align*}
Conversely suppose that $T$ defines an operator in $\mathscr{L}(L^{1}(\om))$, i.e., there exists a constant $M>0$ such that $\|Tf\|_{L^{1}(\om)}\le M\|f\|_{L^{1}(\om)}$ for every $f\in L^{1}(\om)$. Fixing $y_{0}\in\om$, define 
\[f_{y_{0},n}(x):=\frac{1}{|B_{1/n}(y_{0})|} \mathcal{X}_{B_{1/n}(y_{0})}(x),\]
where $n$ is taken sufficiently large such that $B_{1/n}(y_{0})\subseteq \om$. Then we have that $f_{y_{0},n}\in L^{1}(\om)$ with $\|f_{y_{0},n}\|_{L^{1}(\om)}=1$ for every $n$ sufficiently large. Using Fatou's Lemma and the fact that $y_{0}$ is a Lebesgue point for $\rho \left( \frac{x-\cdot}{\eta(x)}\right )$ we have
\begin{align*}
M&\ge\liminf_{n\to\infty} \int_{\om} |T f_{y_{0},n}(x)|\,dx\\
&=\liminf_{n\to\infty}\int_{\om} C(x) \left ( \frac{1}{|B_{n}(y_{0})|}\int_{B_{n}(y_{0})} \rho \left (\frac{x-y}{\eta(x)} \right )\,dy\right ) \,dx.\\
&\ge \int_{\om} C(x) \left (\liminf_{n\to\infty} \frac{1}{|B_{n}(y_{0})|}\int_{B_{n}(y_{0})} \rho \left (\frac{x-y}{\eta(x)} \right )\,dy\right)\,dx\\
&=\int_{\om} C(x) \rho \left (\frac{x-y_{0}}{\eta(x)} \right )dx.
\end{align*}
Since $y_{0}$ was chosen arbitrarily, the proof is complete.

\end{proof}

Note that the proof of Proposition \ref{lbl:L1_L1} implies that if $T\in\mathscr{L}(L^{1}(\om))$ and $\|T\|_{\mathscr{L}(L^{1}(\om))}$ is the operator norm of $T$, then it holds that
\begin{equation}\label{oper_norm}
\|T\|_{\mathscr{L}(L^{1}(\om))}=\sup_{y\in\om}\int_{\om} C(x) \rho \left (\frac{x-y}{\eta(x)} \right )\,dx.
\end{equation}

 An initial approach to prove \eqref{L1_nece_suff} is as follows: Since $ \rho\left ((x-y)/\eta(x) \right )\leq \mathcal{X}_{B_{\eta(x)}(x)}(y)$ and $C(x)\simeq 1/\eta(x)^N$ for all $x,y\in \Omega$, then if
\begin{equation}\label{integrability_suff_pre}
\sup_{y\in\om} \int_{\om} \frac{1}{\eta(x)^{N}} \mathcal{X}_{B_{\eta(x)}(x)}(y)\,dx<\infty,
\end{equation}
it follows that \eqref{L1_nece_suff} holds true. Therefore, a first attempt 
would be to investigate whether we can construct a smooth function $
\theta$ vanishing on $\partial \om$ such that 
\begin{equation}\label{eta_inverse_int}
\int_{\om} \frac{1}{\theta(x)^{N}}\,dx<\infty,
\end{equation}
and use this function $\theta$ instead of $\eta$ in the construction of the operator $T$. We briefly discuss how such a function $\theta$ can indeed be constructed: In \cite{mazya} it was shown that for a bounded, open set $\om$, whose boundary $\partial \om$ has zero Lebesgue measure and upper Minkowski dimension strictly less then $N-\alpha$, for some $0<\alpha<N$,  it holds that
\begin{equation}\label{dist_inverse_int}
\int_{\om} \frac{1}{\mathrm{dist}(x,\partial \om)^{\alpha}}\,dx\,<\infty.
\end{equation}
Then we set $\theta:=(T\sigma)^{\alpha/N}$, and it is not hard to show that $\theta\in C^{\infty}(\om)\cap C(\overline{\om})$ with $\theta=0$ on $\partial \om$, satisfying also \eqref{eta_inverse_int}. We omit the details of the proof here but it resembles closely the  steps of the proof of Proposition \ref{lbl:eta_vanish_linearly} below. However, this approach has a few disadvantages. One of those is that we must assume some regularity for $\partial \om$. Moreover, note that $\theta$ decreases in a sublinear way close to $\partial \om$, hence  it does not hold that $\theta<\sigma$  which implies that  $B_{\theta(x)}(x)$ are not subsets of $\om$ anymore. One can show that if we define the operator $T$ using this smooth function $\theta$, the boundary values of the functions will still be preserved under this operation, but the same will not be true for derivatives. This is a significant disadvantage of this approach as it hinders the possibility to deal with Sobolev spaces, and in particular to obtain trace-preservation results like the one in Proposition \ref{lbl:trace_W1p}.


As a consequence, in what follows, we will construct a smooth function that satisfies \eqref{L1_nece_suff} and vanishes on the boundary at a quadratic rate, resembling more the structure of the original $\eta$ function.
 We note that the following intermediate result can also be derived by \cite{Calderon:1952nx} but we give here an independent proof which employs the operator $T$.  
 \newtheorem{eta_vanish_linearly}[eta_existence]{Proposition}
 \begin{eta_vanish_linearly}[Regularised distance to the boundary function]\label{lbl:eta_vanish_linearly}
 For every $0<\epsilon<1$, there exists a function $\tilde{\eta}\in C^{\infty}(\om)\cap C(\overline{\om})$ with \begin{equation}\label{eta_coercive}
 (1-\epsilon) \sigma(x) \le  \tilde{\eta}(x)\le (1+\epsilon)  \sigma(x), \quad \text{for every }x\in\om.
 \end{equation}
 \end{eta_vanish_linearly}
 
 \begin{proof}
Notice first that in view of \eqref{eta_less_dist}, after some rescaling, the  smooth function $\eta$ of Theorem \ref{lbl:eta_existence} with $\Theta=\partial \om$, can be chosen to satisfy
\begin{align}
\eta(x)&\le \epsilon \sigma(x),\label{eta_eps_dist}\\
|\nabla \eta (x)|&\le \epsilon, \label{nabla_eta_eps}
\end{align}
for every $x\in\om$. 
Hence, we have
\begin{equation}\label{oneminuseps}
1-\epsilon\le 1-\frac{\eta(x)}{\sigma(x)}=\frac{\sigma(x)-\eta(x)}{\sigma(x)}=\frac{\min \{\sigma(y):\;y\in B_{\eta(x)}(x)\}}{\sigma(x)},
\end{equation}
which implies that
\begin{equation}\label{diff_pos}
\sigma(y)-(1-\epsilon)\sigma(x) \ge 0, \quad \text{for every } y\in B_{\eta(x)}(x). 
\end{equation}
We now set $\tilde{\eta}(x)=T\sigma (x)$, where the operator $T$ is defined using the function $\eta$ from above. Then,  from Proposition \ref{lbl:C_boundary_preservation}, $\tilde{\eta}\in C^{\infty}(\om)\cap C(\overline{\om})$, and because of \eqref{diff_pos} we also have that for every $x\in\om$
\begin{align*}
(1-\epsilon)\sigma(x)
&=C(x) \int_{\om} \rho\left (\frac{x-y}{\eta(x)}  \right )(1-\epsilon)\sigma(x)\,dy=\\
&=C(x) \int_{\om} \rho\left (\frac{x-y}{\eta(x)}  \right )\left [\left ((1-\epsilon)\sigma(x)-\sigma(y)\right )+\sigma(y)\right]\,dy\\
&\le C(x) \int_{\om} \rho\left (\frac{x-y}{\eta(x)}  \right )\sigma(y)\,dy\\
&=\tilde{\eta}(x).
\end{align*}
Thus the left inequality of \eqref{eta_coercive} is satisfied. For the right inequality notice that using \eqref{eta_eps_dist} and the fact that $\sigma$ is Lipschitz with constant one, we get
\begin{align*}
\tilde{\eta}(x)
&\le\sigma(x)+|\tilde{\eta}(x)-\sigma(x)|=\sigma(x)+|T\sigma (x)-\sigma(x)|\le \sigma(x)+\sup_{y\in B_{\eta(x)}(x)}|\sigma (x)-\sigma(y)|\\
&\le  \sigma(x)+\sup_{y\in\om:\, |x-y|\le \epsilon \sigma(x)} |\sigma (x)-\sigma(y)|\le \sigma(x)+\epsilon\sigma(x)=(1+\epsilon)\sigma(x),
\end{align*}
ending the proof.
 \end{proof}

 Note that by squaring \eqref{eta_coercive} and rearranging the constants we get the existence of a $C^{\infty}(\om)$ function $\eta:=\tilde{\eta}^{2}$ (not to be confused with the initial $\eta$ used in the proof above) that satisfies 
 \begin{equation}\label{eta_bounded_quadratic}
 \kappa\sigma(x)^{2}\le \eta(x)\le \sigma(x)^{2},\quad \text{for every }x\in\om,
 \end{equation}
 for some $0<\kappa<1$. Observe that the operator $T$ can be also defined  using this $\eta$ and everything proven so far still holds for this case as well. In the next proposition we show that provided $\eta$ satisfies the condition \eqref{eta_bounded_quadratic}  then \eqref{L1_nece_suff} holds and hence $T\in\mathscr{L}(L^{1}(\om))$. In fact we will show that the resulting sequence $(\|T_{n}\|_{\mathscr{L}(L^{1}(\om))})_{n\in\NN}$ is uniformly bounded. 
 
\newtheorem{main_integrability_lemma}[eta_existence]{Theorem}
\begin{main_integrability_lemma}\label{lbl:main_integrability_lemma}
Let $\om\subseteq \RR^{N}$ open and $\eta\in C^{\infty}(\om)$ with the property that there exists $\kappa>0$ such that \eqref{eta_bounded_quadratic} holds.
Then
\begin{equation}\label{integrability_full}
\sup_{n\in\NN}\;\sup_{y\in\om} \int_{\om} C_{n}(x)\rho \left (\frac{x-y}{\frac{1}{n}\eta(x)} \right )dx <\infty,
\end{equation} 
and hence, $T$ and $T_n$ for $n\in\mathbb{N}$ belong to $\mathscr{L}(L^1(\Omega))$.

\end{main_integrability_lemma}

\begin{proof}
Note that it suffices to show
\begin{equation}\label{integrability_suff}
\sup_{n\in\NN}\;\sup_{y\in\om} \int_{\om} \frac{n^{N}}{\eta(x)^{N}} \mathcal{X}_{B_{\frac{\eta(x)}{n}}(x)}(y)\,dx<\infty.
\end{equation}
Note moreover that  the map $x\mapsto C(x)\rho \left (\frac{x-y}{\eta(x)} \right )$ is smooth of compact support. Thus it suffices to show \eqref{integrability_suff} only for $y$'s that belong to a small neighbourhood of the boundary. For that, we will work with $y\in\om$ such that 
\begin{equation}\label{strip}
\sigma(y)<\frac{1}{8}.
\end{equation} 
In fact we may assume, via a rescaling argument, that \eqref{strip} holds for every $y\in\om$.
Using the  layer cake representation formula and Fubini's theorem, we get for every $n\in\NN$ and $y\in\om$
\begin{align}
\int_{\om} \frac{n^{N}}{\eta(x)^{N}} \mathcal{X}_{B_{\frac{\eta(x)}{n}}(x)}(y)\,dx
&= \int_{\om} \left (\int_{0}^{\infty} \mathcal{X}_{\left\{s\in\om:\; \frac{n^{N}}{\eta(s)^{N}}\ge t \right\}}(x)dt \right ) \mathcal{X}_{B_{\frac{\eta(x)}{n}}(x)}(y)\, dx\nonumber\\
&= \int_{0}^{\infty}\left ( \int_{\om} \mathcal{X}_{\left\{s\in\om:\; \frac{n^{N}}{\eta(s)^{N}}\ge t \right\}}(x) \mathcal{X}_{B_{\frac{\eta(x)}{n}}(x)}(y)     \,dx \right )\, dt\nonumber\\
&= \int_{0}^{\infty} \mathcal{L}^{N}\Bigg (\underbrace{\left \{ x\in\om: \frac{\eta(x)}{n}\le \frac{1}{\sqrt[N]{t}}\;\; \text{and} \;\;y\in B_{\frac{\eta(x)}{n}}(x)  \right\}}_{H_{n,y}^{t}} \Bigg )  dt.\label{integral_all_t}
\end{align}
We thus have to estimate the measure of $H_{n,y}^{t}$. Let us for convenience set
\[A_{n,y}:=\left \{x\in\om:\;  y\in B_{\frac{\eta(x)}{n}}(x)  \right \}.\]
Then, from \eqref{eta_bounded_quadratic}   we have
\begin{align*}
A_{n,y}& \subseteq \left \{ x\in\om:\; y\in B_{\frac{\sigma(x)^{2}}{n}}(x) \right \}
\subseteq \left \{ x\in\om:\; |x-y|\le   \frac{\sigma(x)^{2}}{n} \right \}.
\end{align*}
Fixing a $y\in\om$, we choose a $\xi_{y}\in \partial \om$ such that
\[|y-\xi_{y}|=\sigma(y).\]
Hence for every $x\in A_{n,y}$ we have
\[|x-y|\le \frac{\sigma(x)^{2}}{n}\le \frac{|x-\xi_{y}|^{2}}{n}. \]
Setting $z=x-y$, the last inequality can be rewritten as
\begin{align*}
|z|&\le \frac{|z+y-\xi_{y}|^{2} }{n}\le  \frac{(|z|+|y-\xi_{y}|)^{2}}{n}
= \frac{1}{n} \left (|z|^{2}+2 |z||y-\xi_{y}| + |y-\xi_{y}|^{2} \right )\; \Rightarrow\\
0&\le \frac{1}{n} \left (|z|^{2}+ (2|y-\xi_{y}|-n)|z| + |y-\xi_{y}|^{2} \right ).
\end{align*}
By solving the latter quadratic equation with respect to $|z|$ we get
\[|z|\le \frac{n}{2}- |y-\xi_{y}| -\frac{1}{2} \sqrt{n^{2}-4n|y-\xi_{y}|}:=r_{n}(y)\quad \text{or} \quad |z|\ge \frac{n}{2}- |y-\xi_{y}| +\frac{1}{2} \sqrt{n^{2}-4n|y-\xi_{y}|}.\]
Notice, however, that we can ignore the second case above as one can also check that in this case we would have $|x-y|\ge \frac{n}{2}$ and thus $x$ cannot belong to $A_{n,y}$. Also, due to \eqref{strip}, the quantity inside the square root is positive, independently of $n$. Hence we have that if $x\in A_{n,y}$ then $x\in B_{r_{n}(y)}(y)$. Notice furthermore that if $x\in B_{r_{n}(y)}(y)$ then $\sigma(x)\le r_{n}(y)+|y-\xi_{y}|$. Thus we have 
\[\frac{\eta(x)}{n}\le \frac{\sigma(x)^{2}}{n}\le\frac{1}{n}\left (\frac{n}{2}-\frac{1}{2} \sqrt{n^{2}-4n|y-\xi_{y}|}\right )^{2}:=\gamma_{n}(y).\]
In other words, so far we have shown
\[A_{n,y}\subseteq \left \{ x\in\om:\; \frac{\eta(x)}{n}\le \gamma_{n}(y) \right \}.\]
On the other hand, for any $x\in B_{r_{n}(y)}(y)$ and for any point $\xi\in \partial \om$, we have
\begin{equation}\label{lb_start}
|x-\xi|\ge |y-\xi_{y}|-r_{n}(y).
\end{equation}
Indeed if that was not the case then there would exist a $x\in B_{r(y)}(y)$ and $\xi\in \partial \om$ such that
\begin{align*}
|x-\xi|&< |y-\xi_{y}|-r_{n}(y)\; \Rightarrow\\
|x-\xi|+r_{n}(y)&< |y-\xi_{y}|\; \Rightarrow\\
|x-\xi|+|x-y|&< |y-\xi_{y}|\; \Rightarrow\\
|y-\xi|&<|y-\xi_{y}|,
\end{align*}
which is a contradiction from the definition of $\xi_{y}$.
Notice that again due to \eqref{strip}, $|y-\xi_{y}|-r_{n}(y)>0$, and again this is independent of $n$. Now since \eqref{lb_start} holds for every $\xi\in\partial \om$ we get
\begin{align*}
\sigma(x) &\ge |y-\xi_{y}|-r_{n}(y)\; \Rightarrow\\
\sigma(x)^{2} &\ge \left ( |y-\xi_{y}|-r_{n}(y)\right )^{2}\; \Rightarrow\\
\frac{\eta(x)}{n}&\ge \frac{\kappa}{n} \left (-\frac{n}{2}+2|y-\xi_{y}|+\frac{1}{2}\sqrt{n^{2}-4n|y-\xi_{y}|} \right )^{2}:=\beta_{n,\kappa}(y).
\end{align*}
Hence we have showed that 
\begin{equation}\label{bounds_on_eta}
A_{n,y}\subseteq \left \{ x\in\om:\; \beta_{n,\kappa}(y)\le \frac{\eta(x)}{n}\le \gamma_{n}(y) \right \}.
\end{equation}
Notice now that if $t\le \frac{1}{\gamma_{n}(y)^{N}}$, i.e., $\gamma_{n}(y)\le \frac{1}{\sqrt[N]{t}}$, then
\begin{align*}
H_{n,y}^{t}&=\left \{ x\in\om: \frac{\eta(x)}{n}\le \frac{1}{\sqrt[N]{t}}\;\; \text{and} \;\;y\in B_{\frac{\eta(x)}{n}}(x)  \right\}\\
&\subseteq  \left \{ x\in\om: \frac{\eta(x)}{n}\le \gamma_{n}(y)\le \frac{1}{\sqrt[N]{t}} \;\;\text{and} \;\;y\in B_{\frac{\eta(x)}{n}}(x)  \right\}\\
&\subseteq \{x\in \om: x\in B_{\gamma_{n}(y)}(y)\}.
\end{align*}
Hence,
\begin{align}
\int_{0}^{ \frac{1}{\gamma_{n}(y)^{N}}}  \mathcal{L}^{N}\Bigg (\left \{ x\in\om: \frac{\eta(x)}{n}\le \frac{1}{\sqrt[N]{t}}\;\; \text{and} \;\;y\in B_{\frac{\eta(x)}{n}}(x)  \right\} \Bigg )  dt
& \le \int_{0}^{ \frac{1}{\gamma_{n}(y)^{N}}} \omega_{N} \gamma_{n}(y)^{N} dt\nonumber\\
&\le \omega_{N},\label{integral_small_t}
\end{align}
where  $\omega_{N}=\frac{\pi^{N/2}}{\Gamma\left(\frac{N}{2}+1\right)}$ is the volume of the unit ball in $\RR^{N}$. Notice moreover that if $t> \frac{1}{\beta_{n,\kappa}(y)^{N}}$, i.e., $\beta_{n,\kappa}(y)> \frac{1}{\sqrt[N]{t}}$, then
from \eqref{bounds_on_eta} we get that $H_{n,y}^{t}=\emptyset$. Finally we calculate
\begin{align}
\int_{\frac{1}{\gamma_{n}(y)^{N}}}^{\frac{1}{\beta_{n,\kappa}(y)^{N}} }  \mathcal{L}^{N}\Bigg (\left \{ x\in\om: \frac{\eta(x)}{n}\le \frac{1}{\sqrt[N]{t}}\;\; \text{and} \;\;y\in B_{\frac{\eta(x)}{n}}(x)  \right\} \Bigg )  dt
&\le \int_{\frac{1}{\gamma_{n}(y)^{N}}}^{\frac{1}{\beta_{n,\kappa}(y)^{N}} }  \mathcal{L}^{N}\left ( B_{\frac{1}{\sqrt[N]{t}}}(y) \right )  dt\nonumber\\
&=\omega_{N} \int_{\frac{1}{\gamma_{n}(y)^{N}}}^{\frac{1}{\beta_{n,\kappa}(y)^{N}} } \frac{1}{t}\,dt\nonumber\\
&=\omega_{N} N \ln\left (\frac{\gamma_{n}(y)}{\beta_{n,\kappa}(y)} \right ).\label{integral_large_t}
\end{align}
We now compute the ratio $\frac{\gamma_{n}(y)}{\beta_{n,\kappa}(y)}$, where for notational convenience we set $|y-\xi_{y}|=\alpha$. We have
\begin{align*}
\sqrt{\frac{\gamma_{n}(y)}{\beta_{n,\kappa}(y)} }
&=\frac{1}{\sqrt{\kappa}} \frac{\frac{n}{2} -\frac{1}{2} \sqrt{n^{2}-4n\alpha}}{-\frac{n}{2} +2\alpha +\frac{1}{2} \sqrt{n^{2}-4n\alpha}}\\
&=\frac{1}{\sqrt{\kappa}} \frac{n-\sqrt{n^{2}-4n\alpha}}{4\alpha-n+\sqrt{n^{2}-4n\alpha}}\\
&=\frac{1}{\sqrt{\kappa}}  \frac{(n-\sqrt{n^{2}-4n\alpha})((4\alpha-n)-\sqrt{n^{2}-4n\alpha})}{(4\alpha-n)^{2} -(n^{2}-4n\alpha)}\\
&=\frac{1}{\sqrt{\kappa}}\frac{ \sqrt{n^{2}-4n\alpha}}{(n-4\alpha)}\\
&=\frac{1}{\sqrt{\kappa}}\frac{\sqrt{n}}{\sqrt{n-4\alpha}},
\end{align*}
and hence
\begin{equation}\label{ratio_gamma_beta}
\frac{\gamma_{n}(y)}{\beta_{n,\kappa}(y)}=\frac{1}{\kappa} \frac{n}{n-4\alpha}<\frac{2}{\kappa}.
\end{equation}
By combining \eqref{integral_all_t} with \eqref{integral_small_t}, \eqref{integral_large_t} and \eqref{ratio_gamma_beta} we get
\begin{align}
\int_{\om} \frac{n^{N}}{\eta(x)^{N}} \mathcal{X}_{B_{\frac{\eta(x)}{n}}(x)}(y)\,dx
&\le \omega_{N}+\omega_{N}N \ln\left (\frac{\gamma_{n}(y)}{\beta_{n,\kappa}(y)} \right ) \nonumber\\
&=\omega_{N}+\omega_{N}N \ln\left (\frac{1}{\kappa} \frac{n}{n-|y-\xi_{y}|}\right )\label{final_est_n}\\
&\le \omega_{N}+\omega_{N}N \ln\left (\frac{2}{\kappa} \right )\label{final_est_no_n}.
\end{align}
Hence \eqref{integrability_suff} and thus \eqref{integrability_full} holds and the proof is complete.

\end{proof} 

Observe that by using the estimates in the proof of Proposition \ref{lbl:L1_weak}, as well as \eqref{oper_norm} and \eqref{final_est_n}, we can now obtain the following bound on $\|T_{n}\|_{\mathscr{L}(L^{1}(\om))}$:
\begin{align*}
\|T_{n}\|_{\mathscr{L}(L^{1}(\om))}=\sup_{y\in\om} \int_{\om} C_{n}(x)\rho \left (\frac{x-y}{\frac{1}{n}\eta(x)} \right )dx
&\le \sup_{y\in\om}\;  M_{\rho}\int_{\om} \frac{ n^{N}}{ \eta(x)^{N}} \mathcal{X}_{B_{\frac{\eta(x)}{n}}(x)} (y)\, dx\\
&\le M_{\rho}\, \sup_{y\in\om}\; \left (\omega_{N}+\omega_{N}N \ln\left (\frac{1}{\kappa} \frac{n}{n-|y-\xi_{y}|}\right ) \right )\\
&\le M_{\rho}\, \omega_{N}\left (1+ N\ln \left (\frac{1}{\kappa} \frac{n}{n-\tfrac{1}{8}} \right )\right ), 
\end{align*}
something that implies that
\begin{equation}\label{limsup_L1norms}
\limsup_{n\to\infty} \|T_{n}\|_{\mathscr{L}(L^{1}(\om))}\le M_{\rho}\, \omega_{N} \left (1+ N \ln\left (\frac{1}{\kappa} \right) \right ).
\end{equation}

\noindent
\emph{Note}: It still remains unanswered if the estimate \eqref{integrability_full}  (or \eqref{integrability_suff})  holds without imposing any condition of the type \eqref{eta_bounded_quadratic}. We leave that as an open question.

For the following results that deal with $L^{1}$ integrability, we will always assume that the operator $T$ is defined via a function $\eta$ that satisfies \eqref{eta_bounded_quadratic}. We next state the corresponding approximation result in $L^{1}(\om)$:

\newtheorem{convergence_L1}[eta_existence]{Proposition}
\begin{convergence_L1}\label{lbl:convergence_L1}
Let $f\in L^{1}(\om)$. Then $T_{n}f\to f$ in $L^{1}(\om)$ as $n\to\infty$.
\end{convergence_L1}
\begin{proof}
The proof follows exactly the same steps as the one of Proposition \ref{lbl:convergence_Lp}, by taking advantage of the fact that the sequence $(\|T_{n}\|_{\mathscr{L}(L^{1}(\om))})_{n\in\NN}$ is uniformly bounded.
\end{proof}

In view now of Proposition \ref{lbl:convergence_L1} and from the lower semicontinuity of the operator norm with respect to pointwise convergence we have 
\begin{equation}\label{oper_norm_lsc}
1=\|Id\|_{\mathscr{L}(L^{1}(\om))}\le \liminf_{n\to\infty} \|T_{n}\|_{\mathscr{L}(L^{1}(\om))}.
\end{equation}
Thus, the natural question that arises is whether we can further sharpen the upper bound in \eqref{limsup_L1norms} to be one, as result that would also imply that $\lim_{n\to\infty} \|T_{n}\|=1$. In the following paragraphs we describe how we can indeed achieve this, by making some minor amendments to the construction of $T_{n}$.

The first step is to define a sequence of positive, compactly supported, smooth functions $(\rho_{n})_{n\in\NN}$, such $0\le \rho_{n}(x)\le 1$ if and only if $|x|\le 1$, satisfying also
\begin{equation}\label{rhon_condition}
\rho_{n}(x)=1\quad \text{for every $x$ with}\quad |x|\le 1-\frac{1}{n}. 
\end{equation}
That is, $\rho_{n}$ converges to the characteristic function of the ball $B_{1}(0)$.
We also define a sequence of smooth functions $(\eta_{n})_{n\in\NN}$ that satisfy the condition
\begin{equation}\label{eta_n}
\left(1-\tfrac{1}{n}\sigma(x) \right )^{2} \sigma(x)^{2}\le \eta_{n}(x)\le \sigma(x)^{2}.
\end{equation}
Notice that these $\eta_{n}$ can be easily constructed, for instance, by using \eqref{oneminuseps} and applying 
the initial operator $T_{n}$ on $\sigma(x)$, i.e, $\eta_{n}:=T_{n}\sigma$. Observe also that in this case the uniform norm of the gradient of $\eta_{n}$ is uniformly bounded in $n$.

 We can now defined a slightly modified sequence of operators as follows:
\begin{equation}\label{tilde_Tn}
\tilde{T}_{n}f(x):=\tilde{C}_{n}(x) \int_{\om} \rho_{n} \left (\frac{x-y}{\frac{1}{n}\eta_{n}(x)}\right ) f(y)\,dy, \quad x\in\om,\; n\in\NN, \; f\in L_{loc}^{1}(\om), 
\end{equation}
where
\begin{equation}\label{tilde_Cn}
\tilde{C}_{n}(x):=\left (\int_{\om} \rho_{n}  \left (\frac{x-y}{\frac{1}{n}\eta_{n}(x)}\right ) dy \right )^{-1}=\frac{M_{\rho_{n}}}{\eta_{n}(x)^{N}}
\end{equation}
with
\begin{equation}\label{mrn}
M_{\rho_{n}}\le\frac{1}{\omega_{N} \left (1-\frac{1}{n} \right )^{N}}.
\end{equation}

Notice that all the results that we have shown for the operator $T$ (also regarding the convergence of the sequence $T_{n}$) hold for the $\tilde{T}_{n}$ as well. 
%
By following the same steps as in Theorem \ref{lbl:main_integrability_lemma} and the remarks after we get
\begin{equation}\label{norm_tilde_Tn}
\int_{\om} \tilde{C}_{n}(x) \rho_{n} \left (\frac{x-y}{\tfrac{1}{n}\eta_{n}(x)} \right ) dx\le  \frac{1}{\left (1-\frac{1}{n} \right )^{N}} \left (1+N\ln\left (\frac{1}{\left(1-\tfrac{1}{\sqrt{n}}\sqrt{\gamma_{n}(y)}\right)^{2}} \frac{\gamma_{n}(y)}{\beta_{n}(y)} \right ) \right ).
\end{equation}
Since $ \frac{\gamma_{n}(y)}{\beta_{n}(y)}$ and $\tfrac{1}{\sqrt{n}}\sqrt{\gamma_{n}(y)}$ converge to one and zero respectively as $n\to\infty$, also independently of $y$, then we obtain
\begin{equation}\label{limsup_L1_Ttilde}
\limsup_{n\to\infty}\|\tilde{T}_{n}\|_{\mathscr{L}(L^{1}(\om))}=\limsup_{n\to\infty} \sup_{y\in\om} \int_{\om} \tilde{C}_{n}(x) \rho_{n} \left (\frac{x-y}{\tfrac{1}{n}\eta_{n}(x)} \right ) dx\le 1.
\end{equation}


\subsection{The case $\boldsymbol{W^{1,1}(\om)}$}
 Up to this point, the mathematical machinery for the $L^1(\Omega)$ was developed on the assumption \eqref{eta_bounded_quadratic} that imposes simply that $\eta\simeq \sigma^2$. This allows us to handle  $W^{1,1}(\om)$ and subsequently $\bv(\om)$.

\newtheorem{W11_bound}[eta_existence]{Proposition}
\begin{W11_bound}\label{lbl:W11_bound}
The operator $T$ belongs to $\mathscr{L}(W^{1,1}(\om))$.
\end{W11_bound}
\begin{proof}
Note first that since the new function $\tilde{\eta}$ has been constructed via the (initial) operator $T$ acting on $\sigma\in W^{1,\infty}(\om)$, then in view of \eqref{lbl:W1p_bound} we have that $\eta\in W^{1,\infty}(\om)$ as well. If now $f\in W^{1,\infty}(\om)$ then, cf. \eqref{bound_derivative}, the following estimate holds
\begin{equation}\label{bound_derivative_tilde_eta}
|\nabla Tf(x)|\le |T(\nabla f)(x)|+ |\nabla \eta(x)| T(|\nabla f|)(x)\le  |T(\nabla f)(x)|+\|\nabla \eta\|_{\infty} T(|\nabla f|)(x), \quad x\in\om.
\end{equation}
Hence, we have for a generic constant $M$
\begin{align*}
\|\nabla Tf\|_{L^{1}(\om)}\textcolor{red}{=} \int_{\om}  |\nabla Tf(x)|\,dx
&\le \int_{\om} |T(\nabla f)(x)|\,dx + M \int_{\om} T(|\nabla f|)(x)\,dx\\
&\le M \int_{\om}|\nabla f (x)|\,dx= M \|\nabla f\|_{L^{1}(\om)}. 
\end{align*}
In the end we get $\|Tf\|_{W^{1,1}(\om)}\le M\|f\|_{W^{1,1}(\om)}.$
\end{proof}

The analogue of Propositions \ref{lbl:trace_W1p} holds in the case of $W^{1,1}(\om)$, as well. We state them here without proofs as these are similar to the setting with $W^{1,p}(\om)$, where $p>1$.
\newtheorem{trace_W11}[eta_existence]{Proposition}
\begin{trace_W11}\label{lbl:trace_W11}
Suppose that $\om$ has Lipschitz boundary and $f\in W^{1,1}(\om)$. Then
$\tau f=\tau (Tf)$.
\end{trace_W11}

\newtheorem{convergence_W11}[eta_existence]{Proposition}
\begin{convergence_W11}\label{lbl:convergence_W11}
Let $f\in W^{1,1}(\om)$. Then $T_{n}f\to f$ in $W^{1,1}(\om)$ as $n\to\infty$.
\end{convergence_W11}

\subsection{The case $\boldsymbol{\bv(\om)}$}
In this section we examine how the operator acts on $\bv$ functions. We note that also in this case the operator $T$ is defined via the function $\eta$ that satisfies the extra condition \eqref{eta_bounded_quadratic}. 

\newtheorem{BV_bound}[eta_existence]{Proposition}
\begin{BV_bound}\label{lbl:BV_bound}
The operator $T$ belongs to $\mathscr{L}(\bv(\om))$.
\end{BV_bound}
\begin{proof}
Since we have shown already that $T$ belongs to $\mathscr{L}(L^{1}(\om))$, it suffices to show that there exists a constant $M>0$ such that
\[\|\nabla Tf\|_{L^{1}(\om)}\le M |Df|(\om),\quad \text{for all }f\in \bv(\om).\]
Using \eqref{Tf_derivative_bv} we estimate,  for a generic constant $M$,
\begin{align*}
\int_{\om} |\nabla Tf(x)|\, dx
&\le  \int_{\om} C(x)\left |\int_{\om} \rho \left ( \frac{x-y}{\eta(x)} \right )\, dDf(y) \right |\,dx\\
&\quad+  \|\nabla \eta\|_{\infty} \int_{\om} C(x)\left | \int_{\om} \rho \left ( \frac{x-y}{\eta(x)} \right ) \frac{y-x}{\eta(x)}\cdot dDf(y) \right |\,dx\\
&\le  \int_{\om} C(x) \int_{\om} \rho \left ( \frac{x-y}{\eta(x)} \right )\, d|Df|(y) \,dx\\
&\quad+  \|\nabla \eta\|_{\infty}   \int_{\om} C(x) \int_{\om} \rho \left ( \frac{x-y}{\eta(x)} \right )\, d|Df|(y) \,dx\\
&\le  M \int_{\om} \left (\int_{\om} C(x) \rho \left ( \frac{x-y}{\eta(x)} \right ) dx\right ) \, d|Df|(y)\\
& \le M \int_{\om} d|Df|(y)= M |Df|(\om),
\end{align*}
where we have made use of the fact that \eqref{L1_nece_suff} holds.

\end{proof}

Finally, we state the corresponding approximation results in the $\bv$ case. We note that by using the operator $T_{n}$, this approximation is achieved in the sense of weak$^{\ast}$ convergence in $\bv(\om)$. However, by using the alternative operator $\tilde{T}_{n}$, as discussed at the end of Section \ref{sec:L1_bound}, this approximation is achieved in the sense of strict convergence in $\bv(\om)$.
\newtheorem{convergence_BV}[eta_existence]{Proposition}
\begin{convergence_BV}\label{lbl:convergence_BV}
Let $f\in \bv(\om)$. Then the following approximation results holds:
\begin{enumerate}
\item Provided $\om$ has Lipschitz boundary, $T_{n}f\to f$ weakly$^{\ast}$ in $\bv(\om)$ as $n\to\infty$. 
\item $\tilde{T}_{n}f\to f$, strictly in $\bv(\om)$ as $n\to\infty$.
 \end{enumerate}
\end{convergence_BV}
\begin{proof}
We already know from Proposition \ref{lbl:convergence_L1} that $T_{n}f\to f$ in $L^{1}(\om)$. Note that from the estimates in the proof of Proposition \ref{lbl:BV_bound} we obtain
that the sequence $(T_{n}f)_{n\in\NN}$ is bounded in $\bv(\om)$. Then from \cite[Prop. 3.13]{AmbrosioBV} we get that $(T_{n}f)_{n\in\NN}$ converges to $f$ weakly$^{\ast}$ in $\bv(\om)$.

For the second part, notice initially that since we also have  $\tilde{T}_{n}f\to f$ in $L^{1}(\om)$, from the lower semicontinuity of the total variation we also have 
\begin{equation}\label{tv_liminf}
|Df|(\om)\le \liminf_{n\to\infty} \int_{\om} |\nabla \tilde{T}_{n}f(x)|\, dx.
\end{equation}
Moreover, the estimates in the proof of Proposition \ref{lbl:BV_bound}, actually provide 
\begin{equation*}
\int_{\om} |\nabla \tilde{T}_{n}f(x)|\, dx\le \|\tilde{T}_{n}\|_{\mathscr{L}(L^{1}(\om))} \left(1+ \frac{1}{n} \|\nabla \eta_{n}\|_{\infty}\right) |Df|(\om).
\end{equation*}
From the fact that the uniform norm of $\eta_{n}$ is uniformly bounded, and from \eqref{limsup_L1_Ttilde} 
we obtain
\begin{equation}\label{tv_limsup}
\limsup_{n\to\infty}  \int_{\om} |\nabla T_{n}f(x)|\, dx\le  |Df|(\om).
\end{equation}
Combining, \eqref{tv_liminf} and \eqref{tv_limsup} we obtain the strict approximation result.

\end{proof}

\section{Applications: Density of sets}\label{sec:density}

Constraint subsets $K$ of an arbitrary Banach space $X$ often appear in many fields and problem instances in mathematics. These constraints arise as natural limitations of problem variables (as, e.g., in the obstacle problem where a membrane is deflected by acting forces, but not allowed to penetrate a rigid obstacle; here the problem variable is the deflection out of an equilibrium position) and Fenchel dualization of convex optimization problems, among others.  A natural question that appears in the aforementioned settings is the following one: Given a subspace $Y$ densely and continuously embedded in $X$, does the density relation
\begin{equation}\label{eq:Density}
\overline{K\cap Y}^X= K
\end{equation}
holds true, as well. In particular, it is known that a dense and continuous embedding $Y\hookrightarrow X$ is not enough for the above density statement to hold; see \cite{MR3306389}. 

We are interested in the case where $X$ is a Banach space of functions $f:\Omega\to\mathbb{R}^M$, with  $\Omega\subseteq \mathbb{R}^N$ an arbitrary open set, and where $K_G$ is defined as
\begin{equation}\label{eq:K}
K_G=\left \{f\in X: |(Gf)(x)|\leq \alpha(x) \text{ for almost every } x\in \Omega \right\},
\end{equation}
for $G\in \mathscr{L}(X,Z)$ with $Z$ a subspace of $\mathbb{R}^P$-valued maps  with domain in $\Omega$, and $\alpha:\Omega\to \mathbb{R}$ some non-negative function. An important example is given by $G:=\nabla$ the weak gradient, and $X\subseteq W^{1,p}(\Omega)$, for $1<p<\infty$, and $Z\subseteq L^p(\Omega)^N$. If $G=Id$, then we simply write $K=K_{G}$. The difficulty on establishing whether smooth functions are dense on $K_G$ is given by the fact that $\alpha$ is not necessarily a constant and may have a non-trivial zero set. For example, $\alpha$ may vanish in regions of the boundary, and the interior of $\Omega$, see Figure \ref{fig:fig1}.

\begin{figure}
\includegraphics[scale=.4]{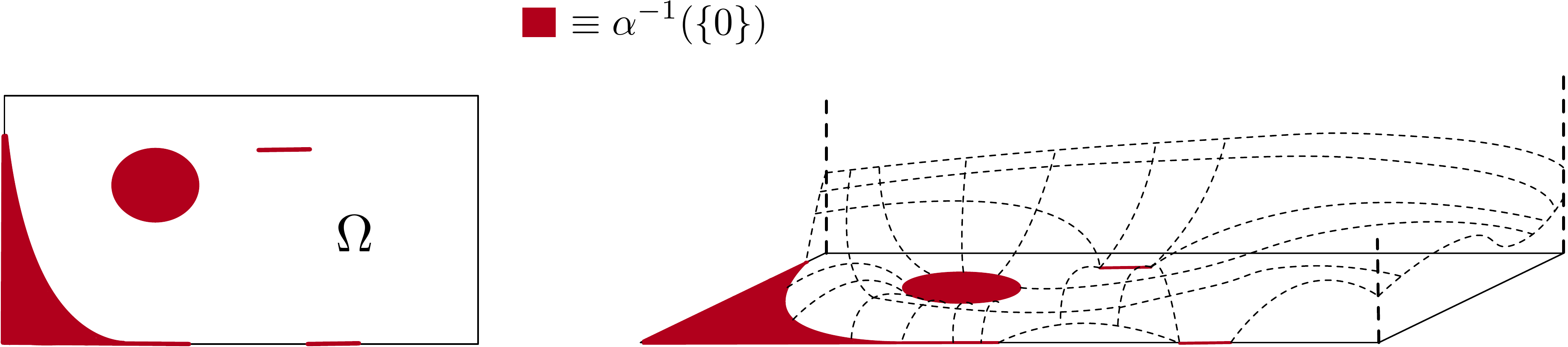}
\caption{Zero set of a possible $\alpha:\Omega\to \mathbb{R}$ (left) and function values (right) }
\label{fig:fig1}
\end{figure}

We consider the following notation throughout this section: Let $\Lambda\subseteq \overline{\Omega}$ be closed, then we define $C_{\Lambda}^\infty(\Omega)$ to be the space of functions $f:\Omega\to \mathbb{R}$ such that $f|_{\Omega\setminus \Lambda}\in C^\infty(\Omega\setminus \Lambda)$ vanishing on $\Lambda$, i.e., $f\in C_{\Lambda}^\infty(\Omega)$ if for
for each $\epsilon>0$, there exists a compact set $\Lambda_\epsilon \subseteq \overline{\Omega}\setminus \Lambda$ such that $(\overline{\Omega}\setminus \Lambda)\setminus \Lambda_\epsilon \subseteq \Omega$ and 
\begin{equation*}
\sup_{x\in (\overline{\Omega}\setminus \Lambda)\setminus \Lambda_\epsilon} |f(x)|\leq \epsilon,
\end{equation*}
and if $x\in \Omega\cap \Lambda$, then $f(x)=0$. In particular, if $\Lambda= \emptyset$, then $C_{\Lambda}^\infty(\Omega)= C^\infty(\Omega)$ and if $\Lambda= \partial\Omega$, then $C_{\Lambda}^\infty(\Omega)= C^\infty(\Omega)\cap C_0(\Omega)$ where $C_0(\Omega)$ is the space of continuous functions vanishing on $\partial\Omega$ (see \cite{attouch2014variational}). Further, if a map $f:\Omega\to \mathbb{R}$ satisfies that $f|_{\Omega\setminus\Lambda}$ is $C^{\infty}$, then (abusing notation) we write $$f\in C^\infty(\Omega\setminus \Lambda).$$ Notice that if $\alpha\in C(\overline{\om})$, $\alpha\ge 0$ and $\Delta=\alpha^{-1}(\{0\})$, then in general it holds $C^{\infty}(\om)\cap K\subsetneq C_{\Delta}^{\infty}(\om)\cap K$ since any function on the latter set is not necessarily differentiable on $\partial \Delta\cap \om$.

We are now ready to state and prove the main density results of this section.

\begin{theorem}\label{thm:density}
Let $\alpha\in C(\overline{\Omega})$ be non-negative and let $\Delta\subseteq\overline{\Omega}$ be defined as $\Delta:=\alpha^{-1}(\{0\})$.

\begin{enumerate}
\item If $G=Id$ and $X=L^p(\Omega)$, with $1<p<\infty$,  it holds true that
\begin{equation*}
\overline{K\cap C_{c}^\infty(\Omega)}^X= K. 
\end{equation*}

\item If $G=Id$, then for $X=W_\Gamma^{1,p}(\Omega)$, with $\Gamma\subseteq\partial\Omega $, it holds true that
\begin{equation*}
\overline{K\cap C_{\Delta}^\infty(\Omega)}^X= K.
\end{equation*}

\item If $G=\nabla$, then for $X=W_{\Gamma}^{1,p}(\Omega)$,  {with $\Gamma\subseteq\partial\Omega $},
it holds true that
\begin{equation*}
\overline{K_G\cap C^\infty(\Omega\setminus\Delta)}^X= K_G.
\end{equation*}
\end{enumerate}
\end{theorem}

\begin{proof}
Note that since $\alpha$ is continuous, $\Delta=\alpha^{-1}(\{0\})$ is closed. Let $\eta^0:=\sum_{i=1}^\infty\eta_i^0$ be the function obtained in  Theorem \ref{lbl:eta_existence} for $\Theta=\partial \om \cup \Delta$.
 Recall that $\supp (\eta_i^0)= \overline{B_i}$ where $(B_i)_{i\in\NN}$ is a sequence of open balls compactly supported in  $\Omega\setminus \Theta$, each $\eta_i^0$ has bounded derivatives of order up to $i$ by the constant $2^{-i}$ and we can also suppose that $\eta^0\leq 1$.

\textit{Step 1}:   We first show that we can construct a smooth function $\eta:\overline{\Omega}\to\mathbb{R}$ such that $\eta(x)=0$ if and only if $x\in \Theta$ and 
\begin{equation}\label{eq:Mn}
M_n(x):= 
\begin{cases}
\frac{\sup_{\{y\in \Omega: |x-y|\leq \eta(x)/n\}}\alpha(y)}{\alpha(x)}, & \text{ if }x\in \overline{\om}\setminus \Theta,\\
1, & \text{ if } x\in \Theta,
\end{cases}
\end{equation} 
satisfies $M_n\to 1$ uniformly in $C(\overline{\Omega})$. 

Let $\omega:[0,\infty)\to [0,\infty)$ denote the modulus of continuity of $\alpha$, i.e.,
\begin{equation*}
|\alpha(z_1)-\alpha(z_2)|\leq \omega(|z_1-z_2|), \quad \text{for all } z_1,z_2\in \overline{\Omega},
\end{equation*}
where $\omega$ is strictly increasing, $\lim_{t\downarrow 0}\omega(t)=\omega(0)=0$, and infinitely differentiable in $(0,\infty)$. Then, for $x\in \overline{\Omega}\setminus \Theta$, we have
\begin{equation*}
|M_n(x)-1|\leq \frac{\sup_{\{y\in \Omega: |x-y|\leq \eta(x)/n\}}|\alpha(y)-\alpha(x)|}{\alpha(x)}\leq  \frac{\omega(\eta(x)/n)}{\alpha(x)}=:H_n(x).
\end{equation*}
Note that $H_{n+1}(x)\leq H_n(x)$ since $\omega$ is increasing. We now show that we can choose $\eta$ such that
\begin{equation}\label{eq:EtaConditions}
\omega(\eta(x))\leq \alpha(x), \qquad \text{and}\qquad \omega(\eta(x))/\alpha(x)\to 0, \quad\text{ as }\quad \overline{\Omega}\setminus \Theta \ni x\to z\in \Theta.
\end{equation}
Let $\Lambda_n$ be defined as
\begin{equation*}
\Lambda_n:=\{x\in \overline{\Omega}\setminus \Theta: 1/(n+1)< d(x, \Theta)\leq 1/n\},
\end{equation*}
for $n\in\mathbb{N}$ and $\Lambda_0:=\{x\in \overline{\Omega}\setminus \Theta: \;1< d(x, \Theta)\}$. Note that each $B_{i}$ is contained in a finite union of $\Lambda_n$ that we denote by $\Theta_i:=\cup_{j=1}^{R_i} \Lambda_{n_j}$. Define the constants
\begin{equation*}
\delta_i:=\inf_{z\in \Theta_i} \omega^{-1}(\alpha(z)\eta^{0}(z))>0,
\end{equation*}
and let 
\begin{equation}\label{eq:etasp}
\eta (x):=\sum_{i=1}^\infty \delta_i \eta^0_i(x).
\end{equation}
If $x\in\Lambda_n$, then in the above sum for $\eta(x)$ we only need to consider $i$ terms where $x\in B_i$, and hence
\begin{equation*}
\eta(x)\leq \sum_{j=1}^\infty \delta_{i_j} 2^{-i_j}\leq \omega^{-1}(\alpha(x)\eta^0(x))  \sum_{j=1}^\infty 2^{-i_j}\leq \omega^{-1}(\alpha(x)\eta^0(x)),
\end{equation*}
and thus $\omega(\eta(x))\leq \alpha(x)\eta^0(x)$ for every $x\in \overline{\om}$. Hence, note that $H_n$ can be continuously extended to $\overline{\Omega}$, still denoted by $H_{n}$, by setting $H_{n}(x)=0$ for every $x\in \Theta$. The same can be done for $M_{n}$, by setting $M_{n}(x)=1$ for every $x\in\Theta$.  Finally, since $(H_n)_{n\in\NN}$ is a non-increasing continuous sequence that converges pointwise to zero, by Dini's theorem, $H_n\to 0$ uniformly, and thus $M_{n}\to 1$ uniformly, as well.

\textit{Step 2:  $G=Id$.} 
 Let $f\in K$ with $f\in W_\Gamma^{1,p}(\Omega)$, and let $T_n$ be the singular mollifier constructed with regulating function $\eta$ in \eqref{eq:etasp}. Since $f\in W_\Gamma^{1,p}(\Omega)$, then $T_{n}f\in W_\Gamma^{1,p}(\Omega)$ and $T_{n}f\to f$ in $W_\Gamma^{1,p}(\Omega)$. Since $f\in K$, then
\begin{equation*}
|T_{n}f(x)|\leq C_{n}(x) \int_{\om} \rho\left (\frac{x-y}{\frac{1}{n}\eta(x)} \right )\alpha(y)\,dy\leq M_n(x)\alpha(x).
\end{equation*}
Consider
\begin{equation*}
\beta_n:=\frac{1}{1+\|M_n-1\|_{\infty}},
\end{equation*}
which implies that $|\beta_n T_{n}f(x)|\leq \alpha(x)$ for all $x\in \Omega$. Since $\beta_n\to 1$, then $\beta_nT_{n}f\to f$ in $W_\Gamma^{1,p}(\Omega)$. Since  $\beta_nT_{n}f\in C_{\Delta}^\infty(\Omega)$, $(ii)$ is proven. 

In order, to prove  $(i)$, consider the following. Let $f\in L^p(\Omega)\cap K$ and define the sequence $(f_m)_{m\in\NN}$ as follows:
\begin{equation*}
f_m(x)=f(x),\quad \text{if }\quad  x\in \left\{y\in \Omega: \;\frac{1}{m}\leq d(y,\Theta)\right\}, \quad \text{and} \quad f_m(x)=0,\quad \text{otherwise}.
\end{equation*}
Note that $f_m\to f$ in $L^p(\Omega)$, and that each $f_m$ has a compact support. The latter implies that for each $m,n\in\mathbb{N}$, we have that $T_{n}f_{m}\in C_c^\infty(\Omega)$. Further, since $f\in K$, then $f_m\in K$ for all $m\in \mathbb{N}$, and with $\beta_n$ as defined  above, it follows that $\beta_nT_{n}f_{m}\in K$ and $\beta_nT_{n}f_{m}\to f_m$ in  $L^p(\Omega)$ as $n\to\infty$. Since $f_m\to f$ in $L^p(\Omega)$, the sequence $(\beta_nT_{n}f_{m_n})_{n\in\NN}$ for some subsequence $(f_{m_n})_{n\in\NN}$ satisfies $\beta_nT_{n}f_{m_n}\to f$ in $L^p(\Omega)$ and $(i)$ is proven.

\textit{Step 3:} Let $G=\nabla$, and $f\in K_G$ with $f\in W_\Gamma^{1,p}(\Omega)$. Then, $T_{n}f\in W_\Gamma^{1,p}(\Omega)\cap C^{\infty}(\om\setminus \Delta)$ and from \eqref{lbl:Tf_derivative}, we have the estimate 
\begin{align*}
|\nabla T_n(f)(x)|
&\leq C_n(x)\int_{\Omega} \rho \left (\frac{x-y}{\frac{1}{n}\eta(x)} \right ) |\nabla f(y)| \, dy\\
&\quad+ C_n(x)\frac{|\nabla \eta(x)|}{\eta(x)}\int_{\Omega} \rho \left (\frac{x-y}{\eta(x)} \right )| y-x | |\nabla f(y)|\, dy\\
&\leq C_n(x)\int_{\Omega} \rho \left (\frac{x-y}{\frac{1}{n}\eta(x)} \right ) \alpha(y) \, dy\\
&\quad+ C_n(x)\frac{|\nabla \eta(x)|}{n}\int_{\Omega} \rho \left (\frac{x-y}{\frac{1}{n}\eta(x)} \right ) \alpha(y) \, dy \\
&\leq \left(1+\frac{\sup_{x\in \Omega}|\nabla \eta(x)|}{n}\right)M_n(x)\alpha(x).
\end{align*}
Define $\tilde{M}_n(x):=\left(1+\sup_{x\in \Omega}|\nabla \eta(x)|/n\right)M_n(x)$ and $\tilde{\beta}_n$ analogously as $\beta_n$ but with $\tilde{M}_n$ instead of $M_n$. Hence, $\tilde{M}_n\to 1$ in the sense of $C(\overline{\Omega})$, and $\tilde{\beta}_n\uparrow 1$, so that $|\beta_n \nabla T_n(f)(x)|\leq \alpha(x)$ for all $x\in \Omega$, and $\beta_n  T_n(f)\to f$ in $X$ which proves $(iii)$.
\end{proof}

\section{Conclusion and outlook} 

We have established analytical properties of the variable step mollifier introduced by Burenkov; in particular, we have determined its boundedness on $L^p$, $W^{1,p}$ and $\bv$ spaces, and proven boundary/interior values preservation properties via the construction of the mollifier. Further, we have shown that the operator is versatile to consider the question of density of smooth functions on certain closed, convex sets in Banach spaces of functions. Future steps concern the application of the mollifier for determination of  constraint qualifications in the context of Fenchel dualization.

\section*{Acknowledgements} 

The authors acknowledge the support of the DFG (German Research Foundation) through the DFG-SPP 1962: Priority Programme ``Non-smooth and Complementarity-based Distributed Parameter Systems: Simulation and Hierarchical Optimization'' within Projects 9 and 10, and under Germany's Excellence Strategy - The Berlin Mathematics Research Center MATH+ (EXC-2046/1, project ID: 390685689) within projects AA4-3 and EF3-5. This work was also supported by the MATHEON Research Center project CH12 funded by the Einstein Center for Mathematics (ECMath) Berlin. KP further acknowledges the financial support of Alexander von Humboldt Foundation.

\appendix
\section{}\label{sec:app}

\begin{proof}[Proof of Proposition \ref{lbl:Tf_derivative}.]
Note first that the product rule gives
\begin{equation}\label{product_rule}
\nabla Tf(x)= \nabla C(x) \int_{\om} \rho\left (\frac{x-y}{\eta(x)} \right ) f(y)\,dy+ C(x) \nabla_{x}\left (\int_{\om} \rho\left (\frac{x-y}{\eta(x)} \right ) f(y)\,dy \right ).
\end{equation}
In what follows we compute a series of gradients, towards a simplification of \eqref{product_rule}.
We first compute the gradient of the map $x\mapsto \rho \left (\frac{x-y}{\eta(x)} \right )$:
\begin{align}
\nabla_{x} \left (\rho\left (\frac{x-y}{\eta(x)} \right ) \right )
&= \nabla \rho \left (\frac{x-y}{\eta(x)}  \right ) \cdot \nabla _{x}\left (\frac{x-y}{\eta(x)}  \right )\nonumber\\
&=  \nabla \rho \left (\frac{x-y}{\eta(x)}  \right ) \cdot \left [\frac{1}{\eta(x)^{2}}\left (\eta(x) I_{N\times N} -(x-y)^{\top}\nabla \eta(x) \right )  \right ]\label{for_altern}\\
&= \nabla \rho \left (\frac{x-y}{\eta(x)}  \right ) \cdot \left ( \frac{1}{\eta(x)} I_{N\times N}- \frac{(x-y)^{\top}}{\eta(x)^{2}} \nabla \eta(x) \right )\nonumber\\
&=\frac{1}{\eta(x)} \nabla \rho \left (\frac{x-y}{\eta(x)}  \right ) \cdot \left (  I_{N\times N}- \frac{(x-y)^{\top}}{\eta(x)} \nabla \eta(x)\right )\nonumber\\
&= -\nabla_{y} \left (\rho\left (\frac{x-y}{\eta(x)} \right ) \right )   \cdot \left (  I_{N\times N}- \frac{(x-y)^{\top}}{\eta(x)} \nabla \eta(x)\right ) \label{gradx_rho}.
\end{align}
Note also that
\begin{equation}\label{grad_C}
\nabla C(x)=-N\frac{M_{\rho}}{\eta(x)^{N}}\frac{\nabla \eta(x)}{\eta(x)}=-N C(x) \frac{\nabla \eta(x)}{\eta(x)}.
\end{equation}
We now argue that we can alternate differentiation and integration, i.e., for every $x\in\om$ we have
\begin{equation*}
\nabla_{x} \left (\int_{\om} \rho \left (\frac{x-y}{\eta(x)} \right )f(y)\,dy \right )=\int_{\om} \nabla_{x}\left ( \rho \left (\frac{x-y}{\eta(x)} \right )\right ) f(y)\,dy. 
\end{equation*}
This can be done, provided that given $x\in\om$ and an open neighbourhood of it $U\Subset \om$, we have that 
\[\nabla_{x}\left ( \rho \left (\frac{x-y}{\eta(x)} \right )\right ) f(y)\le G(y),\quad \text{for every $x\in U$ and $y\in\om$},\]
where $G$ is an integrable function.  Since the expression of interest is zero unless $|y-x|\le \eta(x)$, then in view of \eqref{for_altern}  and the fact that $x$ is away from the boundary, we have
\begin{align*}
\left |\nabla_{x}\left ( \rho \left (\frac{x-y}{\eta(x)} \right ) f(y) \right )\right |
&\le \left |\nabla \rho \left (\frac{x-y}{\eta(x)} \right ) \right | \left(\frac{1}{\eta(x)} + |\nabla \eta (x)|\right) |f(y)|\\
&\le C \mathcal{X}_{U_{\eta}} (y) |f(y)|,
\end{align*}
where $U_{\eta}=\cup_{x\in U} B_{\eta(x)}(x)$. Since $U_{\eta}$ is compactly inside $\om$ we have that $G:=C\mathcal{X}_{U_{\eta}}  f$ is integrable.
Thus, we can again alternate differentiation and integration. Then, using integration by parts, the product rule in $\bv$ and \eqref{gradx_rho}, we get
\begin{align}
 \nabla_{x}\left (\int_{\om} \rho\left (\frac{x-y}{\eta(x)} \right ) f(y)\,dy \right )
 &=\int_{\om} -\nabla_{y}\left (\rho\left (\frac{x-y}{\eta(x)} \right ) \right ) \left (I_{N\times N} + (y-x)^{\top} \frac{\nabla \eta (x)}{\eta(x)} \right ) f(y)\,dy\nonumber\\
 &= \int_{\om} \rho\left (\frac{x-y}{\eta(x)} \right )  \, dDf(y)\nonumber\\
 &\quad+\int_{\om} -\nabla_{y}\left (\rho\left (\frac{x-y}{\eta(x)} \right ) \right )\left ((y-x)^{\top} \frac{\nabla \eta (x)}{\eta(x)} f(y) \right )\,dy\nonumber\\
 &= \int_{\om} \rho\left (\frac{x-y}{\eta(x)} \right ) \, dDf(y)\nonumber\\
 &\quad+ \frac{\nabla \eta (x)}{\eta(x)}\int_{\om} -\nabla_{y}\left (\rho\left (\frac{x-y}{\eta(x)} \right ) \right )\left ((y-x)^{\top} f(y) \right )\,dy\nonumber\\
 &= \int_{\om} \rho\left (\frac{x-y}{\eta(x)} \right ) \, dDf(y)\nonumber\\
 & \quad+N\frac{\nabla \eta (x)}{\eta(x)}\int_{\om} \rho\left (\frac{x-y}{\eta(x)}\right )f(y)\,dy\nonumber\\
 & \quad+ \frac{\nabla \eta (x)}{\eta(x)}\int_{\om} \rho\left (\frac{x-y}{\eta(x)}\right ) (y-x)^{\top}\cdot dDf(y)\label{product_second}
 \end{align}
 Combining now \eqref{product_rule}, \eqref{grad_C} and \eqref{product_second} we get
\begin{align*}
\nabla Tf(x)
&= \nabla C(x) \int_{\om} \rho\left (\frac{x-y}{\eta(x)} \right ) f(y)\,dy+ C(x) \nabla_{x}\left (\int_{\om} \rho\left (\frac{x-y}{\eta(x)} \right ) f(y)\,dy \right )\\
&=-N C(x)  \frac{\nabla \eta(x)}{\eta(x)} \int_{\om} \rho\left (\frac{x-y}{\eta(x)} \right ) f(y)\,dy\\
&\quad  + C(x)\int_{\om} \rho\left (\frac{x-y}{\eta(x)} \right ) dDf(y)\nonumber\\
&\quad +NC(x)\frac{\nabla \eta (x)}{\eta(x)}\int_{\om} \rho\left (\frac{x-y}{\eta(x)}\right )f(y)\,dy\\
&\quad + C(x) \frac{\nabla \eta (x)}{\eta(x)}\int_{\om} \rho\left (\frac{x-y}{\eta(x)}\right ) (y-x)^{\top}\cdot dDf(y)\\
&= C(x)\int_{\om} \rho\left (\frac{x-y}{\eta(x)} \right ) dDf(y)+   C(x) \frac{\nabla \eta (x)}{\eta(x)}\int_{\om} \rho\left (\frac{x-y}{\eta(x)}\right ) (y-x)^{\top}\cdot dDf(y).
\end{align*}

\end{proof}

\bibliographystyle{amsalpha}
\bibliography{kostasbib}
\end{document}